\documentclass[11pt,a4paper,reqno]{article}
\usepackage{amsmath,amssymb,amsfonts,amsthm,hyperref,xcolor}
\usepackage{graphicx,epstopdf,wrapfig}
\usepackage{caption}

\newtheorem*{theorem}{Theorem}
\newtheorem{lem}{Lemma}
\newtheorem*{claim}{Claim}

\DeclareRobustCommand{\qedif}{\hbox{}\nobreak\quad\eqno\hbox{\qedsymbol}}

\title{About the Jones--Wolff Theorem on the Hausdorff~dimension of harmonic measure}
\author{Juli\`a Cuf\'{\i},  Xavier Tolsa, and Joan Verdera}
\date{}

\begin{document}

\maketitle

\section{Introduction}\label{sec1}

These notes are an account on a reading seminar on the Jones--Wolff Theorem on the dimension of harmonic measure, held at the Universitat Aut\`{o}noma de Barcelona during the fall of 2017.  We found that some passages of the original paper~\cite{JW} were somehow hermetic, so that reconstructing the proofs in detail became a delicate task. There is also a presentation of the Jones--Wolff Theorem in chapter 9 of the book ``Harmonic Measure" by Garnett and Marshall~\cite{GM}.  We have expanded the original exposition and also that of the book to cover most of the details. We hope that this will be useful to other readers of this deep and beautiful result.
\vspace{0.3cm}

Let~$K$ be a compact subset of the plane and let $\operatorname{Cap}(K)$ be its logarithmic capacity. Then $\operatorname{Cap}(K)=e^{-\gamma_{K}}$, where $\gamma_{K}$ is the minimal energy of a unit mass distribution on~$K$. The energy of a compactly supported positive Borel measure $\mu$ is
\begin{equation*}\label{energy}
\int \int  \log \frac{1 }{|z-w|} \,d\mu(z) \,d\mu(w)
\end{equation*}
and 
\begin{equation*}\label{wienercap}
\gamma_{K} = \inf \int_K \int_K  \log \frac{1 }{|z-w|} \,d\mu(z) \,d\mu(w),
\end{equation*}
where the infimum is taken over all positive  Borel measures supported on $K$ having total mass $1$, that is,
over all probability measures supported on $K$.  It is a classical result of potential theory that the previous infimum is indeed a minimum and that there is a unique measure attaining the minimum, called the equilibrium measure.  The number  $\gamma_{K} $ is also called the Robin constant of $K$. For example, the capacity of a closed disc is the radius.
\vspace{0.3cm}

Assume  that $\operatorname{Cap}(K)>0$ and set $\Omega=\mathbb{C}^{*}\backslash K$, $\mathbb{C}^{*}$ being the Riemann sphere. Let $\omega(\cdot)=\omega (\Omega, \cdot, \infty)$ be the harmonic measure of~$\Omega$ with pole at infinity. If $\Omega$ is regular with respect to the Dirichlet problem (for example, if it has  smooth boundary)  harmonic measure is defined as follows. Take a continuous function $f$ on the boundary~$\partial \Omega$ of $\Omega$ and let $u$ be the solution of the Dirichlet problem on $\Omega$
with boundary values~$f$. The mapping $ f \rightarrow u(\infty)$ is linear and bounded on the space of continuous functions on~$\partial \Omega$  and hence, by Riesz representation theorem,  there is  a measure $\omega$ on~$\partial \Omega$  satisfying
$$
u(\infty) = \int_{\partial \Omega} f(z) \,d\omega(z), \quad f \in C(\partial \Omega).
$$
This positive measure is the harmonic measure of $\Omega.$  The same can be done replacing the point at infinity with any point $z$ in $\Omega$ and one obtains a harmonic measure $d \omega_{z}$ satisfying
$$
u(z) = \int_{\partial \Omega} f(w) \,d\omega_z(w), \quad f \in C(\partial \Omega).
$$

If the domain $\Omega$ is not smoothly bounded one can use a limiting argument due to Wiener to define harmonic measure \cite{GM}.

The harmonic measure of the complement of a disc is the normalized arclength measure on the boundary. If the domain is simply connected then one uses conformal mapping in computing harmonic measure. 

The notation we adopt for harmonic measure of the domain $\Omega$ with respect to infinity is $\omega(E) = \omega(\Omega,E, \infty), \quad E \subset \partial \Omega.$

It is known that $\omega$ is precisely the equilibrium measure of~$K$. The theorem of Jones and Wolff asserts that $\omega$ lives on a set of Hausdorff dimension less than or equal to one.

\begin{theorem}[Jones--Wolff]
Let $K\subset\mathbb{C}$ be a compact set of positive capacity and let $\Omega=\mathbb{C}^{*}\backslash K$. Then there is a set~$F\subset \partial\Omega$ of Hausdorff dimension less than or equal to~$1$ such that $\omega (\Omega,F,\infty)=1$.
\end{theorem}

Let us recall the definition of Hausdorff content associated with a measure function~$h(r)$ (a non-negative increasing continuous function on $[0, \infty)$ satisfying $h(0)=0$). Let $A$ be a subset of the plane. Then the Hausdorff content of $A$ is
\begin{equation*}\label{content}
M_h(A) = \inf \, \sum_n  h(r_n),
\end{equation*}
where the infimum is taken over all coverings of $A$ by a sequence (possibly finite) of balls of radius $r_n$.
\vspace{0.3cm}

The proof of the Theorem is reduced to showing a statement that looks a little simpler.

Fix $\varepsilon>0$ and take $h(r)=r^{1+\varepsilon}$ as a measure function. Then it is enough to show
\begin{equation}\label{conteneta}
\emph{For each $\eta>0$ there is a set~$A\subset K$ with $M_{h}(A)<\eta$ and $\omega (K\backslash A)<\eta$.}
\end{equation}

 In fact this implies that for $\eta>0$ there is a set $A\subset K$ with $M_{h}(A)<\eta$ and $\omega (K\backslash A)=0$, which in turn implies that there is $A\subset K$ with $M_{h}(A)=0$ and~$\omega(K\backslash A)=0$.  Now taking $\varepsilon_{n}\to 0$, one gets sets $A_{n}\subset K$ with $M_{h_{n}}(A_{n})=0$ ($h_{n}(r)=r^{1+\varepsilon_{n}}$) and $\omega (K\backslash A_{n})=0.$  Letting  $F=\bigcap\limits_{n}A_{n}$ we have $M_{h_{n}}(F)=0$, for each~$n$, which gives that the Hausdorff dimension of $F$ is less than or equal to one, and  $\omega (K\backslash F)=0$.
 \vspace{0.3cm}

 We proceed now to describe a very rough sketch of the proof.  One makes a reduction to the case in which $K$ is a finite union
of pieces of small diameter and rather well separated.  Then one constructs an auxiliary compact~$K^*$, which is a finite union of closed discs,
using two special modification methods, which one calls  ``the disc construction" and the ``annulus construction".  It is crucial to compare
the harmonic measure associated with $\Omega$ and that associated with the new domain~$\Omega^* = \mathbb{C}^* \setminus{K^*}$.  This is simple for the annulus construction, but much more delicate for the disc construction; Lemma~\ref{lem1} below takes care of this issue. The gradient of the Green function~$g$  of $\Omega^*$ with pole at~$\infty$ can be estimated on some special curves surrounding $K^*$ and contained in level sets of~$g$.  All these ingredients allow to estimate the harmonic measure of~$\Omega$ in terms of the integral of the gradient of $g$ on these curves. Lemma~\ref{lem3} is the main tool to end the proof estimating this integral in the appropriate way. An ingredient in the proof of Lemma~\ref{lem3} yields in the limiting case, assuming $\partial \Omega$ smooth,  the formula
$$
\frac{1}{2\pi}\int_{\partial \Omega}\frac{\partial g}{\partial
n}\, \log |\nabla g|\,ds=\sum_k g
(\xi_{k})+\gamma_{K},
$$
where $g$ is now the Green function of $\Omega$ with pole at $\infty,$ $n$ is the unit exterior normal to $\partial \Omega$ and the $\xi_k$ are the critical points of $g$. An heuristic, informal argument to understand that Jones--Wolff's theorem is plausible proceeds as follows.  Since $g$ vanishes on $\partial \Omega$  
the tangential derivative of $g$ vanishes too and $\nabla g = \frac{\partial g}{\partial
n}.$ The harmonic measure is (in the smooth case) 
$$d\omega(z) =   \frac{1}{2\pi} \, \frac{\partial g}{\partial
n}(z)\,ds.$$ Assume that at the point $z$ the ``dimension" of $\omega$ is $d(z)$, which means that $\omega(B(z,r)) \sim r^{d(z)}$, $B(z,r)$ being the disc of center~$z$ and radius~$r$. Since
$$
\frac{\partial g}{\partial n}(z) = \lim_{r \rightarrow 0} \frac{\omega (B(z,r))}{2 r},
$$
we have
$$
\lim_{r \rightarrow 0}  \frac{1}{2\pi} \int_{\partial \Omega} (d(z)-1) \log(2r) \,d\omega(z) = \sum_k g
(\xi_{k})+\gamma_{K}.
$$
The right hand side is positive and thus the integrand in the left hand side of the preceding identity should be non-negative, that is
$d(z) \le 1$, $z \in \partial \Omega$, and so, $\omega$ lives in a set of dimension not greater than $1$.

The theorem of Jones--Wolff has been improved by T.~H.~Wolff in~\cite{Wolff},  where he proves that the harmonic measure is concentrated on a set of $\sigma$-finite length. Also J.~Bourgain~\cite{Bourgain} proved that the dimension of harmonic measure in $\mathbb{R}^n$ is less than $n-\varepsilon$, where $\varepsilon $ is a small positive number depending on $n$. One cannot take $\varepsilon =1$ in dimensions greater than $2$, as shown by T.~H.~Wolff  in~\cite{Wolff2}.

\section{The disc and the annulus construction}\label{sec2}

Let us start with the disc construction.

\subsection*{Disc construction}

Fix $\varepsilon>0$. Let $Q$ be a square with sides parallel to the axes and side length~$\ell=\ell(Q)$ and set $E=Q\cap K$. Replace $E$ by a closed disc~$B$ with the same center as~$Q$ and radius~$r(B)$ defined by
\begin{equation}\label{radius}
r(B)=\frac{1}{2} \frac{\operatorname{Cap}(E)^{1+\varepsilon}}{\ell^{\varepsilon}} =
\frac{1}{2} \frac{e^{-\gamma_{E}(1+\varepsilon)}}{\ell^{\varepsilon}}.
\end{equation}
So we get a new compact set~$\tilde{K}=(K\backslash E)\cup B$, a new domain~$\tilde{\Omega}=\mathbb{C}^{*}\backslash
\tilde{K}=(\Omega \cup E)\backslash B$ and a new harmonic measure~$\tilde{\omega}(\cdot)=\tilde{\omega}(\tilde{\Omega},\cdot,\infty)$.

Note that $B\subset Q$. In fact, since the capacity of a disc is the radius
\begin{equation*}
\operatorname{Cap}(E) \le \frac{\sqrt{2}}{2}\ell,
\end{equation*}
so that
\begin{equation*}
r(B) \le \frac{1}{2} \frac{\left(\sqrt{2}/2\right)^{1+\varepsilon}\cdot \ell^{1+\varepsilon}}{\ell^{\varepsilon}}
=\frac{\ell}{2} \left(\sqrt{2}/{2}\right)^{1+\varepsilon} \le \ell/2.
\end{equation*}

\subsection*{Annulus construction}

Let $Q$ be a square with sides parallel to the axis and take the square~$RQ$, where $R$ is a number larger than $1$ that will be chosen later. One has to think that $R$ is very large. Delete $K\cap (RQ \backslash Q)^{0}$ from $K$ to obtain a new domain~$\tilde{\Omega}=\Omega\cup (RQ\backslash Q)^{0}$ and a new harmonic measure~$\tilde{\omega}(\cdot)=\omega (\tilde{\Omega},\cdot,\infty)$.

\vspace*{7pt}

It is important to have some control on the harmonic measure of the new domain obtained after performing the disc or the annulus construction. For the annulus this is easy:  any part of~$K$ which has not been removed has larger or equal harmonic measure. In other words,
if~$A \subset \partial 	\Omega$ satisfies $A\cap (RQ\backslash Q)=\emptyset$, then $\tilde{\omega}(A)\ge \omega(A)$. This is a consequence of the fact that $A\subset\partial\Omega\cap \partial\tilde{\Omega}$ and $\Omega\subset \tilde{\Omega}$ (the domain increases and the set lies in the common boundary).

\vspace*{7pt}

Estimating the harmonic measure after the disc construction is a difficult task. The result is the following.

\begin{lem}\label{lem1}
Let $Q$ be a square with sides parallel to the axis. Fix $\varepsilon>0$ and perform the disc construction for this~$\varepsilon$. Assume that $RQ\backslash Q\subset \Omega$. Then there exists  a number~$R_0(\varepsilon)$ such that for $R\ge R_0(\varepsilon)$  one has
\begin{enumerate}
\item[a)] $\tilde{\omega}(B)\ge C(\varepsilon) \, \omega(Q\cap K)$, where $C(\varepsilon)$ is a positive constant depending only on~$\varepsilon$.
\item[b)] $\tilde{\omega}(A)\ge \omega(A)$, if $A\subset \partial\Omega\backslash RQ$.
\end{enumerate}
\end{lem}

The proof of Lemma~\ref{lem1} will be presented in Section~\ref{sec6} and we will use it as a black box along the paper.

\section{Domain modification}\label{sec3}

Let $\Omega=\mathbb{C}^{*}\backslash K$, $\operatorname{Cap} K>0$ and assume that $K\subset \{|z|<1/2\}$ (this assumption will be convenient later on, but it is not essential; we can reduce to this situation by a dilation). Fix $\varepsilon>0$ and let $R>2+R_0(\varepsilon)$, $R$ integer, where $R_0(\varepsilon)$ is the constant given by Lemma~\ref{lem1}.  We let $M$ stand for a large constant that will be chosen later and we let $\rho$ be a small constant so that $M\le \log 1/\rho$, and $\rho=\frac{1}{2^{N}}$, $N$~a positive integer. Consider the grid~$\mathcal{G}$ of dyadic squares of side length~$\rho$ and lower left corner at the points of the form $\{(m+ni)\,\rho : \, m,n\in \mathbb{Z} \}$. For each $1\le p,q\le R$, let $\mathcal{G}_{pq}$ be the family of squares~$Q\in \mathcal{G}$ with $(m,n)\equiv (p,q)$ (mod $R\times R$). Then $\mathcal{G}=\bigcup\limits_{p,q=1}^{R} \mathcal{G}_{pq}$.

Write $K_{pq}=\bigcup\limits_{Q\in \mathcal{G}_{pq}}K\cap Q$, $\Omega_{pq}=\mathbb{C}^{*}\backslash K_{pq}$, $\omega_{pq}(A)=\omega(\Omega_{pq},A,\infty), \;A	\subset \partial \Omega_{pq}= \partial K_{pq}.$  If $A\subset \partial\Omega= \partial K$, then $\omega (A)=\sum\limits_{pq}\omega (A\cap K_{pq})\le \sum\limits_{pq}\omega_{pq}(A\cap K_{pq})$ because $\Omega\subset \Omega_{pq}$ and so $A\cap K_{pq}\subset\partial \Omega\cap \partial \Omega_{pq}$.  Note that it is enough to prove the theorem for $K=K_{pq}.$  Indeed if we know the result for all $K_{pq}$, then there are sets $A_{pq}\subset K_{pq}$ with $\dim_{H}(A_{pq})\le 1$ and $\omega_{pq}(K_{pq}\backslash A_{pq})=0$. Taking $F=\bigcup\limits_{pq}A_{pq}$ we get
$\dim_{H}(F)\le 1$ and
\begin{equation*}
\begin{split}
\omega(K\backslash F)&=\omega \left(\bigcup_{pq}K_{pq}\backslash \bigcup_{pq}A_{pq}\right)\\*[9pt]
&=\omega \left(\bigcup_{pq}(K_{pq}\backslash A_{pq})\right)\le \sum_{pq}\omega (K_{pq}\backslash A_{pq})\le
\sum_{pq}\omega_{pq}(K_{pq}\backslash A_{pq})=0.
\end{split}
\end{equation*}
We remark that, by construction, for each square~$Q\in \mathcal{G}_{pq}$ one has $RQ\backslash Q\subset \Omega_{pq},$ so that
we are entitled to apply Lemma \ref{lem1}.

From now on we fix $p$, $q$ and let $\Omega=\Omega_{pq}$, $K=K_{pq}$, $\omega=\omega_{pq}$.  We let $\{Q_j\}_j$ be the family of squares in $\mathcal{G}_{pq}.$

\vspace*{7pt}

Fix $\varepsilon>0$ and perform the disc construction for $\varepsilon$  in every square~$Q_{j}$, so that we get a finite family of closed discs~$\{B_{j}\},$  whose union is a compact set $K_1$,  a new domain~$\Omega_{1}= \mathbb{C}^* \setminus K_1$  and a new harmonic measure~$\omega_{1} (\cdot)= \omega(\Omega_1, \cdot, \infty)$.

\vspace*{7pt}

Next choose a dyadic square~$Q^{1}$ of largest side~$\ell(Q^{1})$, not necessarily from $\mathcal{G}_{pq},$ such that
$$
\ell(Q^{1})\ge \rho\quad\text{and}\quad \omega_{1}(Q^{1})\ge M\ell (Q^{1}).
$$
If such $Q^{1}$ does not exist we stop the domain modification. If $Q^{1}$ exists we perform the annulus construction on~$Q^1$ (with constant~$R$) and after this we perform the disc construction on the square~$Q^{1}$, replacing $K_{1}
\cap Q^{1}$ by a disc~$B^{1}$.
So we obtain a new compact $K_2$, a new domain~$\Omega_{2}= \mathbb{C}^* \setminus K_2$ and a new harmonic measure~$\omega_{2} (\cdot)= \omega(\Omega_2, \cdot, \infty)$.

\vspace*{7pt}

Now we continue and take $Q^{2}$ dyadic with largest side such that $Q^{2}\not\subset Q^{1}$, $\ell (Q^{2})\ge \rho$ and $\omega_{2}(Q^{2})\ge M\ell (Q^{2})$. If such $Q^{2}$ does not exist we stop. Otherwise we perform the annulus construction on~$Q^{2}$ but with a special rule: If $B^{1} \cap(\partial (RQ^{2}\backslash Q^{2}))\ne \emptyset$, then we do not remove the set $B^{1}\cap (RQ^{2}\backslash Q^{2})$ from $K_{2}$. The reason for this rule is to get full balls in all cases.

After that we perform the disc construction on~$Q^{2}$, replacing $K_{2}\cap Q^{2}$ by the corresponding disc~$B^{2}$, getting a new compact $K_3,$ a new domain~$\Omega_{3}$ and a new harmonic measure~$\omega_{3}$.

\vspace*{7pt}

We continue this process so that if $K_{1}\cap Q^{1}$, $K_{2}\cap Q^{2},\dotsc, K_{n-1}\cap Q^{n-1}$ have been replaced by $B^{1},\dotsc, B^{n-1}$ we choose now (if it exists) a dyadic cube~$Q^{n}$ with largest side so that
$$
Q^{n} \not\subset Q^{j} ,\quad j=1,\dotsc,n-1,\quad \ell (Q^{n})\ge \rho,\quad \omega_{n}(Q^{n})\ge M\ell (Q^{n}).
$$
Then (if we do not stop) we perform the annulus construction with respect to~$Q^{n}$ but without removing $B^{j}\cap (RQ^{n}\backslash Q^{n})$, $j=1,\dotsc,n-1$ in case that $B^{j}\cap (\partial (RQ^{n}\backslash Q^{n}))\ne\emptyset$ (this is the special rule). Finally we perform the disc construction on~$Q^{n}$, getting $B^{n}$, $K_{n+1}$, $\Omega_{n+1}$ and $\omega_{n+1}$.

At each step there are only finitely many candidate dyadic squares, because $\rho \le  \ell(Q) \le 1/M.$ Since
no $Q^{j}$ can be repeated (because $Q^{j}\not\subset Q^{\ell}$, $\ell=1,\dotsc,j-1$) the modification process stops after finitely many steps. Let $K^{*},\Omega^{*}=\mathbb{C}\backslash K^{*}$, $\omega^{*}(\cdot)=\omega (\Omega^{*},\cdot,\infty)$ be the final outcome so that $K^{*}$ is the disjoint union of the non removed discs; more precisely,
$$
K^{*}=\bigcup_{k\in S}B^{k}\cup \bigcup_{j\in T}B_{j} \quad \text{(some finite sets of indices~$S$ and $T$)},
$$
where the $B_j$ are the discs produced after the first step that have survived all process and the $B^k$ are the new discs produced after performing the annulus and the disc constructions.

Now we want to prove by means of Lemma~\ref{lem1} the following estimates:
\begin{alignat}{2}
 &\omega^{*} (B_{j})\ge C(\varepsilon)\,\omega (Q_{j}), &\quad &j\in T,\label{equ1}\\
 &\omega^{*} (Q^{j})\ge C(\varepsilon) \, M \ell (Q^{j}), &\quad &j\in S.\label{equ2}
\end{alignat}
For \eqref{equ1} note first that  we always have $R Q_{j}\backslash Q_{j}\subset \Omega.$  Since $Q_j$ has survived all steps we cannot have $RQ^{k}\supset Q_{j}$ at some step $k$.  Since $RQ^k$ is a union of dyadic squares, the other possibility is $RQ^{k}\cap Q_{j}=\emptyset$ for all $k.$ This means that we can apply the first inequality of Lemma~\ref{lem1} in the first step and the second in all others.

\vspace*{7pt}

For \eqref{equ2}, when we select $Q^{j}$ we have $\omega_{j}(Q^{j})\ge M \ell (Q^{j})$ and after performing the annulus and the disc constructions, we get $\omega_{j+1}(B^{j})\ge C(\varepsilon)\, \omega_{j} (Q^{j})\ge C(\varepsilon) \, M\ell (Q^{j})$. If $k>j$ there are three possibilities:
i)~$B^{j}\subset RQ^{k}\backslash Q^{k},$ in which case $B^{j}$ has  disappeared and $j$ would not be in $S$; ii)~$B^{j}\cap (RQ^{k}\backslash Q^{k})=\emptyset$ in which case $\omega_{k+1}(B^{j})\ge \omega_{j+1}(B^{j})$ and iii)~$B^{j}\cap \partial (RQ^{k}\backslash Q^{k})\ne \emptyset$.

\begin{figure}[ht]
\begin{center}
\includegraphics{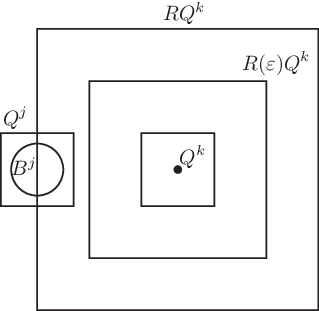}
\end{center}
\end{figure}

\noindent
In this last case we have $\ell (Q^{k})\ge \ell (Q^{j})$ since otherwise $Q^{k}$ would had disappeared. But now since $R=2+R_0(\varepsilon)$ we get that $B^{j}\cap (R_0(\varepsilon)Q^{k}\backslash Q^{k})=\emptyset$ and so $\omega_{k+1}(B^{j})\ge \omega_{j+1}(B^{j})$ by Lemma \ref{lem1} part b). At the end we obtain
$$
\omega^{*} (Q^{j})\ge \omega^{*} (B^{j})\ge \omega_{k+1}(B^{j})\ge \omega_{j+1}(B^{j})\ge C(\varepsilon)\, \omega_{j}(Q^{j})\ge C(\varepsilon) \, M\ell (Q^{j}).
$$

\pagebreak

We will also need the following estimate.

If $z_{0}\in Q_{j}$, $j\in T$ (or $z_{0}\in Q^{k}$, $k\in S$) and $r \ge \ell (Q_{j})$ ($r\ge \ell (Q^{k})$), then
\begin{equation}\label{equ3}
 \omega^{*}\{ |z-z_{0}| <r\} \le CMr.
\end{equation}
Let us discuss the case of $Q_{j}$, $z_{0}\in Q_{j}$. We remark that if $Q$ is a dyadic square with $Q\supset Q_{j}$,
then one has $\omega^{*} (Q)\le M\ell (Q)$ because otherwise the process would not have been stopped.

\vspace*{15pt}

\begin{minipage}[c]{0.35\textwidth}
\raggedright
\includegraphics{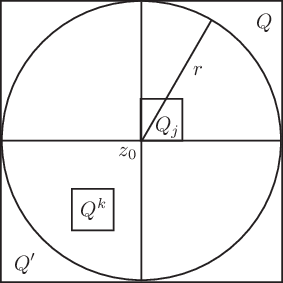}\hfill
\end{minipage}\hspace*{.45cm}
\begin{minipage}[c]{0.55\textwidth}
Take now a dyadic square $Q\supset Q_{j}$ with side length~$2^{m}\ell (Q_{j})$ such that $r \le 2^m \ell(Q_j) \le 2r$.
We just said that $\omega^{*} (Q)\le Mr$. Now the disc~$\{|z-z_{0}|<r \}$ is contained in $4$~dyadic squares of the same side length
as~$Q$. Take one of these squares $Q'$ different from $Q.$ If $Q'$ does not contain any $Q_{j'}$ or $Q^k$ then $\omega^{*}(Q')=0.$
Otherwise $\omega^{*}(Q') \le M r.$

The case $z_{0}\in Q^{k}$ is dealt with similarly.
\end{minipage}

\vspace*{15pt}

The next lemma shows that the union of the family of squares $\{Q_j\}_{j \in T}$ and a dilation of the family $\{Q_k\}_{k \in S}$
contains~$K$.

\begin{lem}\label{lem2}
$K\subset \bigcup\limits_{k\in S}2RQ^{k}\cup \bigcup\limits_{j\in T}Q_{j}$.
\end{lem}

\begin{proof}
Recall that now $K=K_{pq}=\bigcup\limits_{Q\in \mathcal{G}_{pq}}K\cap Q$. So let $Q\in \mathcal{G}_{pq}$ and $E=K\cap Q$. If $Q=Q_{j}$ for some $j\in T$ then $E \subset Q_{j}$ and so $E\subset \bigcup\limits_{k}2RQ^{k}\cup\bigcup\limits_{j\in T}Q_{j}$.

If $Q\ne Q_{j}$ for every $j\in T$ then there is a  first index~$j_{1}$ such that $Q\subset RQ^{j_{1}}\backslash Q^{j_{1}}$;
if~$j_{1}\in S$ then $Q\subset RQ^{j_{1}}$, $j_{1}\in S$, and we are done. If $j_{1}\notin S$ there is  a first index~$j_{2}$
such that $Q^{j_{1}}\subset RQ^{j_{2}}\backslash Q^{j_{2}}$. In this case $\ell (Q^{j_{2}})\ge 2\ell (Q^{j_{1}})$ because if
we had $\ell (Q^{j_{1}})\ge \ell (Q^{j_{2}})$ then $Q^{j_{2}}\subset RQ^{j_{1}}$ and $Q^{j_{2}}\subset R {Q^{j_{1}}}\backslash
Q^{j_{1}}$, so that $Q^{j_{2}}$ would have disappeared. If $j_{2}\in S$ we have $Q\subset R Q^{j_{2}}$ and we are done. If $j_{2}\notin S$ there is a first~$j_{3}$ such that
$$
Q^{j_{2}}\subset RQ^{j_{3}}\backslash Q^{j_{3}}
$$
and so on.

We get a sequence $j_{1}< j_{2}<\dotsb <j_{n}$ with $j_{1},\dotsc,j_{n-1}\notin S$, $j_{n}\in S$ so that
$Q^{j_{k}}\subset RQ^{j_{k+1}}\backslash Q^{j_{k+1}}$ and $ \ell( Q^{j_{i+1}}) \geq 2 \ell (Q^{j_{i}})$, which implies $Q\subset 2RQ^{j_{n}}$.
The double radius appears because we need to argue on two steps: in the first we use that $Q^{j_{n-1}} \subset R Q^{j_n}$ and in the
second that $Q \subset RQ^{j_{n-1}}.$
\end{proof}

\section{Surrounding \boldmath$K^{*}$ by level curves of the Green function}\label{sec4}

To continue the proof of the Theorem, let $Q$ be a square~$Q=Q_{j}$, $j\in T$ or $Q=Q^{k}$, $k\in S$ and let $B$ be
the corresponding disc.
Let~$g(z)=g(\Omega^{*},z,\infty)$ be the Green function of the domain $\Omega^*.$
The goal of this section is  to find a closed curve~$\sigma$ surrounding $B$,  contained in a level set of~$g$,  and such that
\begin{equation}\label{gradsigma}
|\nabla g(z)|\le CM^{2}\log 1/\ell(Q), \quad  z\in \sigma,
\end{equation}
for a positive constant $C.$

The Green function $g$ is the logarithmic potential of the equilibrium measure plus the Robin constant,
that is,
\begin{equation*}
\begin{split}
g(z)&=\int_{K^{*}}\log |z-w|\,d\omega^{*}(w)+\gamma_{K^{*}}\\*[7pt]
&=\int_{B}\log |z-w|\,d\omega^{*}(w)+
\int_{K^{*}\backslash B}\log |z-w|\,d\omega^{*}(w)+\gamma_{K^{*}}=: u(z)+v(z)+\gamma_{K^{*}}.
\end{split}
\end{equation*}
We have the estimate
\begin{equation}\label{equ4}
|\nabla v(z)|\le \int_{K^{*}\backslash B}\frac{d\omega^{*}(w)} {|z-w|} \le CM\log 1/\ell(Q), \quad z \in Q\setminus B.
\end{equation}
To show this inequality, fix $z\in Q\backslash B$ and set $\omega^{*}(t)=\omega^{*}(B(z,t))$. We have
\begin{equation*}
\begin{split}
\int_{K^{*}\backslash B}\frac{d\omega^{*}(w)}{|z-w|} &\le \int^{1}_{\ell(Q)}\frac{d\omega^{*}(t)}{t}\le
\omega^{*}(B(z,1))+\int^{1}_{\ell(Q)}\frac{\omega^{*}(t)}{t^{2}}\,dt \\*[7pt]
&\le
1+CM \int^{1}_{\ell(Q)}\frac{dt}{t}\le 1+CM \log 1/\ell(Q) \le CM\log 1/\ell(Q),
\end{split}
\end{equation*}
where we have used~\eqref{equ3}.

\vspace*{7pt}

We would like  to estimate the derivative~$\frac{\partial u}{\partial r}(z)$ from below.
Assume for simplicity that the center of the square~$Q$, and so of the disc~$B$, is the origin, and write $z=re^{i\theta}$.

Since
$$
u (re^{i\theta})=\frac{1}{2} \int_{B}\log |re^{i\theta}-w|^{2}\,d\omega^{*}(w),
$$
we have
\begin{equation*}
\begin{split}
\frac{\partial u}{\partial r}(z)&=\frac{1}{2}\int_{B}\frac{1}{|re^{i\theta}-w|^{2}}\, \frac{\partial}{\partial r}
\left((re^{i\theta}-w)(re^{-i\theta}-\bar w)\right)\,d\omega^{*}(w)\\*[9pt]
&=\int_{B}\operatorname{Re}
\left( \frac{(z-w)\, \bar z}{|z-w|^{2}\, |z|}\right)\,d\omega^{*}(w),
\end{split}
\end{equation*}
which in particular tells us that  $\frac{\partial u}{\partial r}(z) \ge 0$.

\vspace*{7pt}

Now we write
$$
\operatorname{Re}\left(\frac{(z-w)\,\bar z}{|z-w|^{2}|z|}\right)=\frac{1}{|z-w|} \left\langle \frac{z-w}{|z-w|},\frac{z}{|z|}\right\rangle
$$
and we look for the minimum value of $\left\langle\frac{z-w}{|z-w|},\frac{z}{|z|}\right\rangle$ when $|w|= \tau$, $\tau$ being
 the radius~$r(B)$ of~$B$.

\begin{figure}[ht]
\begin{center}
\includegraphics{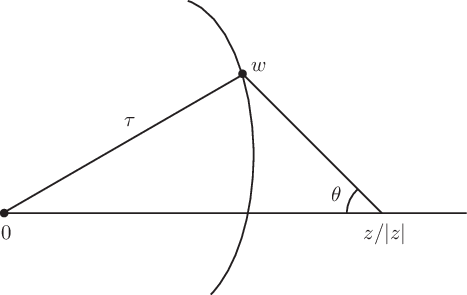}
\end{center}
\vspace*{-15pt}
\end{figure}

Assuming that $\frac{z}{|z|}=1$,  set $\left\langle \frac{z-w}{|z-w|},1\right\rangle=\cos \theta$ (see the figure). The cosine Theorem
yields
$$
\cos \theta=\frac{1}{2|z|}\left(|z-w|+\frac{|z|^{2}-\tau^{2}}{|z-w|}\right)
$$
so that the minimum is attained for
$$
|z-w|=\sqrt{|z|^{2}-\tau^{2}},
$$
that is, when $z-w$ is orthogonal to~$w$.

We then have
$$
\operatorname{Re}\left(\frac{(z-w)\, \bar z}{|z-w|^{2}|z|}\right)\ge \frac{\sqrt{|z|-\tau}\, \sqrt{|z|+\tau}}{|z-w| |z|} \ge
\frac{\sqrt{|z|-\tau}}{|z| \,\sqrt{|z|+\tau}},
$$
and also
$$
\frac{\sqrt{|z|-\tau}}{|z| \,\sqrt{|z|+\tau} }\ge\frac{1}{|z|}\left(1-\frac{\tau}{|z|}\right).
$$
Returning to the case of a square~$Q$ centered at the point~$z_{0}$ with $\tau=r(B)$ we get
the estimate of $\frac{\partial u}{\partial r}(z)$ we are looking for, namely,
\begin{equation}\label{eq5}
\frac{\partial u}{\partial r}(z)\ge
\frac{\sqrt{|z-z_{0}|-r(B)}}{\sqrt{|z-z_{0}|+r(B)}} \, \frac{\omega^{*}(B)}{|z-z_{0}|} \ge \frac{\omega^{*}(B)}{|z-z_{0}|}-
\frac{r(B)\, \omega^{*}(B)}{|z-z_{0}|^{2}}, \quad |z-z_{0}| > r(B).
\end{equation}

We are now ready to estimate the gradient of the Green function $g$. Define
$$
\alpha=\alpha(B)=\max \left(\frac{\omega^{*}(B)} {M^{2}\log 1/\ell(Q)},2r(B)\right)
$$
and distinguish two cases:

\enlargethispage{3mm}

\vspace*{7pt}
\noindent
Case 1:  $\alpha= 2r(B)$, that is, $\dfrac{\omega^{*}(B)}{M^{2}\log 1/\ell(Q)}\le 2r(B)$.

\vspace*{7pt}

We let $\sigma$ to be the circle $\partial B$ so that we need to prove the estimate
$$
|\nabla g(z)|\le CM^{2}\log 1/\ell(Q), \quad z\in \partial B.
$$

\vspace*{7pt}

This is a consequence of the inequality
\begin{equation}\label{eq6}
\sup_{\partial B}|\nabla g|\le C\inf_{\partial B}|\nabla g|
\end{equation}
for some constant~$C$.


In fact, using \eqref{eq6} one gets
$$
\omega^{*}(B)=\frac{1}{2\pi} \int_{\partial B}\frac{\partial g}{\partial n}\,ds\ge \inf_{\partial B}|\nabla g|\, r(B)
$$
and for $z\in \partial B$
\begin{equation}\label{grad3}
|\nabla g(z)|\le \sup_{\partial B}|\nabla g(z)| \le C\inf_{\partial B}|\nabla g(z)|\le C\, \frac{\omega^{*}(B)}{r(B)}\le CM^{2}\log 1/\ell(Q).
\end{equation}
In order to prove \eqref{eq6} assume that $z_{0}=0$ and take two points $z$ and $z'$ with $|z|=|z'|=2r(B)$. Then we have
$$
m^{-1}g(z')\le g(z)\le mg(z')
$$
for some constant~$m$; this follows by applying Harnack's inequality to discs of radius~$\delta < r(B)$ centered at points on the circle
$\{|z|=2 r(B)\}$, chosen so that the discs of radius~$\delta/2$ cover this circle.

Take now $z$ and $z'$  with $r(B)<|z|=|z'| < 2r(B)$. We also have
$$
m^{-1}g(z')\le g(z)\le mg(z').
$$
Indeed, for $\theta\in [0,2\pi]$, write $g_{\theta}(z)=g(e^{i \theta} z)$, then
$$
m^{-1}g_{\theta}(z)\le g(z)\le mg_{\theta}(z)
$$
holds for $|z|=2r(B)$, and trivially also holds for $|z|=r(B)$, $\theta\in [0,2\pi]$. By the maximum principle we get
$$
m^{-1}g_{\theta}(z)\le g(z)\le mg_{\theta}(z),\quad r(B)\le |z|\le 2r(B),\quad \theta\in [0,2\pi].
$$
As a consequence, for $|z|=|z'|=r(B)$ and $n$, $n'$ the unit exterior normal vectors to~$\partial B$ at~$z$ and~$z'$, we have
$$
\frac{m^{-1}g(z'+tn')}{t}\le \frac{g(z+tn)}{t}\le \frac{mg(t'+tn')}{t}
$$
and so
$$
m^{-1} \frac{\partial g}{\partial n}(z')\le \frac{\partial g}{\partial n}(z)\le m \frac{\partial g}{\partial n}(z'),\quad |z|=|z'|=r(B)
$$
and finally $\sup\limits_{|z|=r(B)}|\nabla g|\le C\inf\limits_{|z|=r(B)}|\nabla g|,$ as required.

\vspace*{7pt}
\noindent
Case 2:  $\alpha> 2r(B)$, that is,  $\alpha=\dfrac{\omega^{*}(B)}{M^{2}\log 1/\ell(Q)}$.

\vspace*{7pt}

We note that
\begin{equation}\label{alfapetit}
\alpha\le \frac{\omega^{*}(Q)}{M^{2}\log 2}\le \frac{2 M\ell (Q)}{M^{2}\log 2} \le \frac{4}{M}\, \ell(Q).
\end{equation}
The inequality~$\omega^{*}(Q)\le M\ell (Q)$, for $Q=Q_{j}$, comes from the fact that $Q_j$ has survived the process to get  to~$\omega^{*}$.
If $Q=Q^{k}$, take the dyadic square~$\tilde{Q}$ with side length~$2\, \ell (Q^{k})$ and containing~$Q^{k}$. Since the process has
 stopped,  $\omega^{*}(Q^{k}) \le \omega^{*}(\tilde{Q}) \le M\ell (\tilde{Q}) = 2 M\ell (Q)$.
 
Taking in \eqref{alfapetit} $M > 8 $, we obtain $\alpha \le \ell(Q)/2$ and so  $\{|z-z_{0}|=\alpha\}\subset Q.$

Now we want to prove that
\begin{equation}\label{eq6primera}
|\nabla g(z)| \le  4\,M^{2}\log 1/\ell(Q), \quad   \alpha \le |z-z_{0}|\le \mu\alpha,
\end{equation}
where $\mu$ is such that $\mu > e^{20 \pi},$  a condition that will be used later. Choosing $M >  8 \mu$ we obtain $ \alpha \mu < \ell(Q)/2$, by \eqref{alfapetit}. Hence the annulus $\alpha \le |z-z_{0}|\le \mu\alpha$ is contained in~$Q\setminus B,$ a fact that will be used in the sequel without further mention.

\vspace*{7pt}

Let us show
\begin{equation}\label{derivadau}
\frac{\partial u}{\partial r}(z)\ge |\nabla v(z)|, \quad \alpha \le |z-z_{0}|\le \mu\alpha.
\end{equation}
By \eqref{eq5} we get
$$
\frac{\partial u}{\partial r}(z)\ge
\frac{\sqrt{|z-z_{0}|-r(B)}}{\sqrt{|z-z_{0}|+r(B)}} \, \frac{\omega^{*}(B)}{|z-z_{0}|} \ge \frac{\sqrt{\alpha-r(B)}}{\sqrt{\alpha+r(B)}}\,
\frac{\omega^{*}(B)}{\mu\alpha}, \quad  \alpha < |z-z_{0}| \le  \mu\alpha,
$$
where we have used that the function $x\to \frac{\sqrt{x-r(B)}}{\sqrt{x+r(B)}}$ is increasing.

Since $\alpha>2r(B),$  taking the quotient $M / \mu$ big enough, we have
\begin{equation*}
\begin{split}
\frac{\partial u}{\partial r}(z)&\ge \frac{1}{\sqrt{3}} \, \frac{\omega^{*}(B)}{\mu\alpha} \\*[7pt]
&=\frac{1}{\sqrt{3}\,\mu} \, M^{2} \log 1/\ell(Q)\ge C\, M\log 1/\ell(Q)\ge |\nabla v(z)|, \quad \alpha \le |z-z_{0}|\le \mu\alpha,
\end{split}
\end{equation*}
by \eqref{equ4}.

Therefore
$$
|\nabla g(z) |\le |\nabla u(z) |+|\nabla v(z) |\le 2|\nabla u(z)|\le 2 \int_{\partial B}\frac{d\omega^{*}(w)}{|z-w|}, \quad \alpha \le |z-z_{0}|\le \mu\alpha,
$$
and $|z-w|\ge |z-z_{0}|-|w-z_{0}|\ge \alpha-r(B)\ge \frac{\alpha}{2},$ which gives
$$|\nabla g(z)| \le 4\, \frac{\omega^{*}(B)}{\alpha} =
4 \,M^{2}\log 1/\ell(Q),  \quad \alpha \le |z-z_{0}|\le \mu\alpha,
$$
as required.

\vspace*{7pt}

Assume $z_0=0$, let $c=\sup \{g(z):|z|=\alpha\}$ and take as $\sigma$ the connected component of~$\{g=c\}$ that contains a point on~$|z|=\alpha$. The curve~$\sigma$ encloses a domain that contains the disc  $\{|z| < \alpha\}.$

\smallskip

\begin{center}
\includegraphics{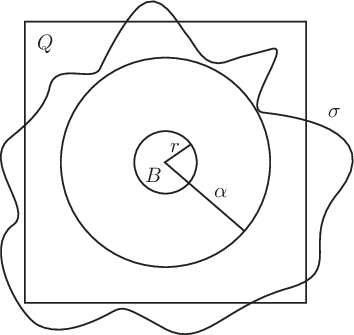}
\end{center}

We claim that $\sigma$ remains inside $\{|z|\le \mu \alpha\},$ which, in view of \eqref{eq6primera}, yields the required estimate \eqref{gradsigma}.

We have
$$
|\nabla u(z)|\le \int_{B}\frac{d\omega^{*}(w)}{|z-w|}\le 2\frac{\omega^{*}(B)}{|z|}, \quad |z|>\alpha,
$$
because
$$
|z-w|\ge |z|-|w| \ge \frac{|z|}{2}+\frac{\alpha}{2}-r(B)>\frac{|z|}{2}.
$$
By \eqref{eq5}
$$
\frac{\partial u}{\partial r}(z)\ge \frac{\omega^{*}(B)}{|z|}-\frac{r(B)\omega^{*}(B)}{|z|^{2}}, \quad |z| > r(B).
$$
Note that
$$
\frac{r(B)\omega^{*}(B)}{|z|^{2}}\le \frac{1}{2}\frac{\omega^{*}(B)}{|z|}
$$
because
$
|z| \ge \alpha \ge 2\,r(B).
$
Then, for $ |z| > \alpha $,
$$
\frac{\partial u}{\partial r}(z)\ge \frac{1}{2}\frac{\omega^{*}(B)}{|z|}\quad\text{and}\quad |\nabla u(z)|\le 4\frac{\partial u}{\partial r}(z).
$$
Therefore, by \eqref{derivadau},
\begin{equation}\label{eq8}
|\nabla g(z)|\le |\nabla  u(z)|+|\nabla v(z)|\le 5\frac{\partial u}{\partial r}(z), \quad \alpha \le |z| \le \mu \alpha.
\end{equation}

\vspace*{7pt}
\begin{minipage}[c]{0.29\textwidth}
\raggedright
\includegraphics{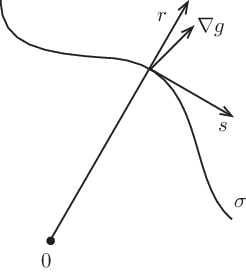}\hfill
\end{minipage}\hspace*{.65cm}
\begin{minipage}[c]{0.61\textwidth}

Note that since the quotient $M/ \mu$ can be taken as large as we want, we can improve \eqref{derivadau} to
\begin{equation*}\label{derivadaum}
\frac{1}{2} \frac{\partial u}{\partial r}(z)\ge |\nabla v(z)|, \quad  \alpha \le |z-z_{0}|\le \mu\alpha.
\end{equation*}
Then
\begin{equation}\label{derivadag}
\begin{split}
\frac{\partial g}{\partial r}(z) &= \frac{\partial u}{\partial r}(z)+\frac{\partial v}{\partial r}(z)   \\*[7pt]
&\ge \frac{\partial u}{\partial r}(z)-|\nabla g(z)| >0, \quad \alpha \le |z| \le \mu \alpha.\!\!
\end{split}
\end{equation}
\end{minipage}

\vspace*{7pt}
\noindent
The curve $\sigma$ contains at least a point $a$ on the circle $\{ |z|  = \alpha\}.$ Consider the maximal subarc $\tau$ of $\sigma$ containing $a$ and contained in the disc $\{ |z|  \le  \mu \alpha\}.$  By \eqref{derivadag},  each ray emanating from the origin intersects $\tau$ only once, and so $\tau$ can be parametrized by the polar angle $\theta$ in the form
$r(\theta) e^{i \theta}$ with  $\theta_1 \le \theta \le \theta _2.$ Without loss of generality assume $ \theta_1 < 0 <  \theta _2$  and $r(0) = a.$

If $\tau = \sigma$ we are done. If not, $r(\theta_2)= \mu \alpha$ and we will reach a contradiction.
If $r$ is the radial direction and $s$~is the orthogonal direction to~$r$, then~\eqref{eq8} yields
$$
\left|\frac{\partial g}{\partial s}(z)\right| \le |\nabla g(z)| \le 5 \frac{\partial u}{\partial r}(z)\le 10 \frac{\partial g}{\partial r}(z).
$$

Since $g(r(\theta)e^{i\theta})=c$, taking the derivative with respect to~$\theta$ one gets
$$
0=\langle \nabla g (r(\theta) e^{i \theta}), r'(\theta) e^{i \theta} + i r(\theta)e^{i \theta} \rangle =
r'(\theta) \frac{\partial g}{\partial r}+ r(\theta)  \frac{\partial g}{\partial s}
$$
that gives
$$
|r'(\theta)| \, \frac{\partial g}{\partial r} =r(\theta) \left| \frac{\partial g}{\partial s}\right|
$$
and so
$$
\frac{|r'(\theta)|}{r(\theta)}\le 10.
$$
Therefore
$$
\log \frac{r(\theta_{2})}{r(0)} =\int^{\theta_{2}}_{0}\frac{r'(\theta)}{r(\theta)}\,d\theta \le \int^{\theta_{2}}_{0}\frac{|r'(\theta)|}{r(\theta)}\,d\theta \le 20\pi
$$
and, recalling the way $\mu$ has been chosen,
$$
r(\theta_{2})\le  e^{20\pi} r(0)= e^{20\pi}  \alpha < \mu \alpha,
$$
which is a contradiction.
 By \eqref{eq6primera}
we obtain the desired inequality \eqref{gradsigma}.

\section{The estimate of the gradient of Green's function on the level curves} \label{sec5}

In the previous section we have exhibited for each disc~$B=B_{j}$, $j\in T$ or $B=B^{k}$, $k\in S$, a simple curve~$\sigma$ contained in a level curve of~$g$ and surrounding $B,$ on which the estimate~\eqref{gradsigma} holds.
Let now $\Gamma$ be the curve formed by the set of $\sigma$'s corresponding to each disc~$B_{j}$ or~$B^{k}$.  Then $\Gamma$ separates $K^{*}$ from infinity.

In this section we prove the estimate
\begin{equation}\label{gradgamma}
\frac{1}{2\pi}\int_{\Gamma}\frac{\partial g}{\partial n}|\mathrm{log}|\nabla g||\,ds \le C \, \log\log (1/\rho).
\end{equation}

Since we are assuming that $M\le \log (1/\rho),$  we have, by \eqref{gradsigma},
$$
\log^{+}|\nabla g (z) |\le \log (CM^{2}\log 1/\ell(Q))\le C \, \log \log (1/\rho), \quad z \in \Gamma.
$$
Note that
$$
\frac{1}{2\pi}\int_{\Gamma}\frac{\partial g}{\partial n}\,ds=\frac{1}{2\pi}\sum_{\sigma}\int_{\sigma}\frac{\partial g}{\partial n}\,ds =\sum_{B}\omega^{*}(B)
$$
which is clear for those terms for which $\sigma=\partial B$ and follows from the divergence theorem for the others, because $\sigma$ surrounds~$\partial B$.

Hence
\begin{equation*}
\begin{split}
\frac{1}{2\pi}\int_{\Gamma}\frac{\partial g}{\partial n}\log^{+}|\nabla g|\,ds &\le  C \, \log\log (1/\rho) \, \frac{1}{2\pi}\int_{\Gamma}\frac{\partial g}{\partial n}\,ds\\*[7pt]
&=C \, \log\log (1/\rho) \, \sum_{B}\omega^{*}(B)\le C \, \log\log (1/\rho).
\end{split}
\end{equation*}

\vspace*{7pt}

In order to estimate the integral on~$\Gamma$ of $\frac{\partial g}{\partial n}\log^{-}|\nabla g|$ we need the following lemma.

\begin{lem}\label{lem3}
Let $g(z)=g(\Omega,z,\infty)$ be the Green function of the domain~$\Omega$ with pole at infinity and let $\Gamma=\bigcup\limits^{N}_{j=1}\Gamma_{j}$ be the union of finitely many closed Jordan curves~$\Gamma_{j}$  so that $\Gamma\subset \{|z|<1\}$, $\Gamma$ separates $K=\mathbb{C}^{*}\backslash \Omega$ from infinity and there are constants $c_{j}$, $j=1,\dotsc,N$ such that $\Gamma_{j}\subset \{g(z)=c_{j}\}$, $j=1,\dotsc, N$. Then
$$
\frac{1}{2\pi} \int_{\Gamma}\frac{\partial g}{\partial n} \, \log |\nabla g|\,ds >-\log 2,
$$
where $n$ is the outward unit normal to~$\Gamma$.
\end{lem}

The proof of this lemma will be discussed in Section~\ref{sec6}.

\vspace*{7pt}

By Lemma~\ref{lem3} we have
$$
\frac{1}{2\pi}\int_{\Gamma}\frac{\partial g}{\partial n} \, \log^{-}|\nabla g|\,ds\le \frac{1}{2\pi}\int_{\Gamma}\frac{\partial g}{\partial n} \, \log^{+}|\nabla g|\,ds+\log 2,
$$
which completes the proof of \eqref{gradgamma}.

\section{End of the proof of the Theorem} 
Recall from Section~\ref{sec1} that for a fixed $\varepsilon>0$ and for each $\eta>0$ we have to find a set~$A\subset K$ with $M_{h}(A) <\eta$ and $\omega(K\backslash A)<\eta$, where $h(r)=r^{1+\varepsilon}$.

\vspace*{7pt}

Decompose the set of indices $T$ as  $T=T_{1}\cup T_{2}$ with
\begin{align*}
T_{1}&=\{j\in T: \omega^{*}(B_{j})\ge \rho^{\varepsilon/2}\, r_{j}\},\\
T_{2}&=\{j\in T: \omega^{*}(B_{j})\le \rho^{\varepsilon/2}\, r_{j}\},
\end{align*}
where $r_{j}=r(B_{j})$.

Set
$$
A=\left[K\cap \left(\bigcup_{k\in S}2RQ^{k}\right)\right]\cup \left[K\cap \left(\bigcup_{j\in T_{1}}Q_{j}\right)\right].
$$
We know, by Lemma~\ref{lem2}, that
$$
K\backslash A=\bigcup_{j\in T_{2}}(K\cap Q_{j}).
$$
In order to prove that $M_{h}(A)<\eta$ we need the following well known estimate.

\begin{lem}\label{lem4}
Let $E$ be a compact set and let $h(r)= r^s$ for some $s>0$. Then there exists a positive constant $C=C(s)$ such that
$$
M_{h}(E)\le C\, (\operatorname{Cap}E)^{s}.
$$
\end{lem}

For sake of completeness we give a proof of this lemma in Section~\ref{sec6}.

\vspace*{5pt}

For $h(r)=r^{(1+\varepsilon)},$ inequality \eqref{equ2} yields, using that $\sum_{k\in S} \omega^*(Q^{k}) \leq 1$,
\begin{equation*}
\begin{split}
M_{h}\left( K\cap \left( \bigcup_{k\in S}2R Q^{k}\right)\right)&\lesssim (2R)^{1+\varepsilon} \, \sum_{k\in S} \ell(Q^{k})^{1+\varepsilon}\\*[5pt]
&\le \frac{R^{1+\varepsilon}}{(M \, C(\varepsilon))^{1+\varepsilon}}\, \sum_{k\in S}\omega^{*}(Q^{k})^{1+\varepsilon}\le \left(\frac{R}{M\, C(\varepsilon)}\right)^{1+\varepsilon} \leq \eta
\end{split}
\end{equation*}
for $M$ big enough.  By Lemma \ref{lem4} with $s=1+\varepsilon$ and the definition of the radius of $B_j$ in the disc construction \eqref{radius} we obtain
\begin{equation*}
\begin{split}
M_{h} \left(\bigcup_{j\in T_{1}} ( K \cap Q_{j} ) \right) &\le \sum_{j\in T_{1}}M_{h}(K \cap Q_{j}) \le C\, \sum_{j\in T_{1}} \operatorname{Cap}(K \cap Q_{j})^{1+\varepsilon}  \\*[5pt]
&=C\sum_{j\in T_{1}} r_{j} \, \rho^{\varepsilon}=C\sum_{j\in T_{1}}r_{j} \, \rho^{\varepsilon/2} \, \rho^{\varepsilon/2}\\*[5pt]
&\le C\sum_{j\in T_{1}}\rho^{\varepsilon/2}\, \omega^{*}(B_{j})\le C\rho^{\varepsilon/2}\le\eta
\end{split}
\end{equation*}
provided $\rho$ is small enough.

We have got $M_{h}(A)<\eta$ and it remains to estimate $\omega (K\backslash A)$.

By inequality~\eqref{equ1}
$$
\omega(K\backslash A)=\omega \left(\bigcup_{j\in T_{2}}(K\cap Q_{j})\right)\le \frac{1}{C(\varepsilon)} \, \sum_{j\in T_{2}}\omega^{*}(B_{j}).
$$
Now we remark that for $j\in T_{2}$ we are in the Case~1 of the Section~\ref{sec4}, that is
$$
\frac{\omega^{*}(B_{j})}{M^{2}\log (1/\rho)}\le 2r_{j}.
$$
Indeed, since $\omega^{*}(B_{j})\le \rho^{\varepsilon/2}r_{j}$ it is enough to see that
$$
\rho^{\varepsilon/2} \le 2 M^{2}\log (1/\rho),
$$
which clearly holds for $\rho$ sufficiently small.

\vspace*{7pt}

For $z\in \partial B_{j}$, $j\in T_{2}$, we know by \eqref{grad3} that
$$
|\nabla g(z)|\le C\, \frac{\omega^{*}(B_{j})}{r_{j}}\le C\, \rho^{\varepsilon/2},
$$
so that
$$
\log |\nabla g(z)|\le \log C + \frac{\varepsilon}{2} \log \rho\ \leq \frac{\varepsilon}{4} \log\rho,
$$
for small enough $\rho.$  Hence, for such small $\rho$,
$$
|\mathrm{log} |\nabla g(z)||\ge \frac{\varepsilon}{4}\log (1/\rho).
$$
We then get
\begin{equation*}
\begin{split}
\omega(K\backslash A)&\le \frac{1}{C(\varepsilon)}\sum_{j\in T_{2}}\omega^{*}(B_{j})\le \frac{1}{C(\varepsilon)}\frac{1}{2\pi}\sum_{j\in T_{2}}\int_{\partial B_{j}}\frac{\partial g}{\partial n}\,ds\\*[5pt]
&\le \frac{C}{C(\varepsilon)\, \varepsilon\log (1/\rho)}\, \sum_{j\in T_{2}}\frac{1}{2\pi}\int_{\partial B_{j}}\frac{\partial g}{\partial n}\, |\mathrm{log} |
\nabla g||\,ds\\*[5pt]
&\le \frac{C}{C(\varepsilon)\, \varepsilon\log (1/\rho)}\, \frac{1}{2\pi}\int_{\Gamma}\frac{\partial g}{\partial n}\, |\mathrm{log} |\nabla g||\,ds\\*[5pt]
&\le \frac{C}{\varepsilon\, C(\varepsilon)}\, \frac{\log\log (1/\rho)}{\log (1/\rho)},
\end{split}
\end{equation*}
due to~\eqref{gradgamma}. Thus $\omega(K\backslash A)<\eta$ if $\rho$ is small enough. Therefore for fixed $\varepsilon>0$ and given $\eta>0$, we can choose $M$ and $\rho$ such that the set~$A$ satisfies the desired conclusion.

\section{Proof of the lemmas}\label{sec6}

\subsection{An auxiliary lemma}

To prove Lemma \ref{lem1} we need the following auxiliary result.

\begin{lem}\label{lemfacil}
Let $F\subset \mathbb C$ be a compact set with positive logarithmic capacity. Then there exists 
a probability measure $\mu$ supported in $F$ such that the potential 
$$U_\mu(z) = \int\log\frac1{|z-w|}\,d\mu(w)$$
is a continuous function in $\mathbb C$.
\end{lem}

\begin{proof}
Let $\sigma$ be the equilibrium measure for $F$, so that $\sigma$ is a probability measure supported on $F$ such that
$$U_\sigma(z) = \int\log\frac1{|z-w|}\,d\sigma(w) \leq \gamma_F<\infty, \quad\text{
for $\sigma$-a.e.\ $z\in F$.}$$
By Egoroff there exists a compact subset $F_0\subset F$ such that $\sigma(F_0) >0$ and
$$\lim_{r\to 0}\int_{|z-w|\leq r}\left|\log\frac1{|z-w|}\right|\,d\sigma(w) =  0,\quad \text{
uniformly on $z\in F_0$.}$$
 It is then clear that
$$\lim_{r\to 0}\int_{|z-w|\leq r}\left|\log\frac1{|z-w|}\right|\,d\sigma|_{F_0}(w) = 0,\quad\text{
uniformly on $z\in F_0$.}$$ 

We claim that $U_{\sigma|_{F_0}}$ is continuous in $\mathbb C$. By the continuity principle of classical Potential Theory it is enough to  show that the restriction of $U_{\sigma|_{F_0}}$ to $F_0$
is continuous in $F_0$. To prove this, for any $\varepsilon \in (0,1/2)$ and $z,z'\in F_0$ such that $|z-z'|\leq 
\varepsilon^2$,
write
\begin{align*}
|U_{\sigma|_{F_0}}(z) - U_{\sigma|_{F_0}}(z')| & \leq \!\int_{|z-w|\leq \varepsilon} \left|\log\frac1{|z-w|}\right|d\sigma|_{F_0}(w) + \!\int_{|z-w|\leq \varepsilon} \left|\log\frac1{|z'-w|}\right|d\sigma|_{F_0}(w) \\
& \quad + \int_{|z-w|> \varepsilon} \left|\log\frac{|z'-w|}{|z-w|}\right| d\sigma|_{F_0}(w).
\end{align*}
The first integral on the right hand side tends to $0$ as $\varepsilon\to0$, and the same happens with the second one, taking into account
that $\{w:|z-w|\leq \varepsilon\}\subset \{w:|z'-w|\leq 2\varepsilon\}$. For the last one, observe that
$$
\frac{|z'-w|}{|z-w|} = 1 + O(\varepsilon), \quad \text{for $w,z,z'$ such that $|z-w|> \varepsilon$ and $|z-z'|\leq
\varepsilon^2$,}$$
and thus, for some constant $C>0$,
$$\int_{|z-w|> \varepsilon} \left|\log\frac{|z'-w|}{|z-w|}\right| d\sigma|_{F_0}(w)
\leq C \varepsilon \sigma(F_0).
$$
Therefore, 
$$\lim_{\varepsilon\to0} \sup_{z,z'\in F_0:|z-z'|\leq \varepsilon^2}|U_{\sigma|_{F_0}}(z) - U_{\sigma|_{F_0}}(z')| =0,$$
which completes the proof of the claim. The lemma follows by taking $\mu=\sigma|_{F_0}/\sigma(F_0)$.
\end{proof}

\subsection{Proof of Lemma~\ref{lem1}}

Changing scale we may assume that $\ell (Q)=1.$ Let $\xi_{0}$ stand for the center of~$Q$.

\begin{proof}[Proof of a)]
By Lemma \ref{lemfacil} there exists a probability
measure $\mu$ supported on $(\Omega\cup E)^c$ such that
the convolution $\log |z| *\mu$ is continuous in $\mathbb C$. Then,
being $\Omega\cup E$ an unbounded domain, 
the Green function~$g(z,\xi)$ of the domain~$\Omega\cup E$ with pole at~$\xi$ can be written in the form
\begin{equation}\label{conda**}
g(z,\xi)=\int \log \frac{|z-a|}{|z-\xi|} \,d\mu(a)
+\iint \log \frac{|w -\xi|}{|w-a|}\,d\mu(a)\,d\omega (\Omega\cup E,w,z), \quad  z \in \Omega \cup E. 
\end{equation}
Indeed, observe 
that, for any fixed $\xi\in \Omega\cup E$, the function defined by the first summand is harmonic in
$(\Omega\cup E\cup\{\infty\})\setminus \{\xi\}$ and continuous up to boundary of $\Omega\cup E$,
and  the 
second summand
solves the Dirichlet problem in $\Omega\cup E$ with continuous boundary data equal to minus the function defined
by the first summand.

Note that both measures $d\mu$ and $d\omega (\Omega\cup E,\cdot,z)$ are supported in 
$\partial\Omega\setminus RQ$.
From \eqref{conda**} it is clear that  the Green function can also be written in the form
\begin{equation}\label{eqgreen*}
g(z,\xi)=\log \frac{1}{|z-\xi|} +h (z,\xi),\quad z\in \Omega\cup E,\quad \xi\in \Omega\cup E,
\end{equation} 
with
\begin{equation}\label{eqghreen*}
h(z,\xi)= \iint \log \frac{|w -\xi|\,|z-a|}{|w-a|}\,d\omega (\Omega\cup E,w,z)\,d\mu(a)
,\quad z\in \Omega\cup E,\quad \xi\in \Omega\cup E
\end{equation}
Clearly
\begin{equation}\label{eq10}
|\nabla_{\xi}h(z,\xi)|\le \left|\int_{\partial \Omega\backslash RQ}\frac{1}{\overline{w}-\overline{\xi}}\,d\omega  (\Omega\cup E,w,z) \right| \le O\left(\frac{1}{R}\right),
\quad \xi\in Q, \quad z \in  \Omega\cup E.
\end{equation}
Next, for a given $z_{0}\in \partial Q$,
 we wish to estimate $h(\xi_{0},z_{0})$ from below. To this end, note that, for all $a\in\operatorname{supp}\mu
 \subset \partial\Omega\setminus RQ$,
 $|z_0-a|\geq \frac12(R-1)\geq R/4\geq \frac12|\xi_0-z_0|$ (because we assume $R\geq2$), and thus, 
 for all $w\in 
\partial\Omega\setminus RQ$, 
$$|w-a|\leq |w-\xi_0| + |\xi_0-z_0| + |z_0-a|  \leq |w-\xi_0| + 3|z_0-a|.$$
Thus,
using the two estimates
$|z_0-a|\geq R/4$ and $|w-\xi_0|\geq \frac12R$, we derive
$$|w-a|\leq |w-\xi_0| \,\frac{|z_0-a|}{R/4} + 3|z_0-a| \,\frac{|w-\xi_0|}{R/2}
= \frac{10|w-\xi_0|\,|z_0-a|}{R}.
$$
Hence,
$$\log \frac{|w -\xi_0|\,|z_0-a|}{|w-a|} \geq \log \frac R{10},\quad w\in\partial\Omega\setminus RQ, 
\quad a\in\partial\Omega\setminus RQ
.$$
Plugging this into \eqref{eqghreen*}, we obtain
\begin{equation}\label{hac}
h(z_0,\xi_0)\geq \log\frac R{10}.
\end{equation}

Let now $\mu_{E}$ and $\mu_{B}$ be the equilibrium measures of~$E$ and $B$ respectively and set
$$
u(z) : =\int_{B} g(z,\xi)\,d\mu_{B}(\xi),\quad
v(z) : =\int_{E}g(z,\xi)\,d\mu_{E}(\xi).
$$
For every $z_{0}\in \partial Q$ one has 
\begin{alignat*}{2}
u(\eta)&=\gamma_{B}+h(z_{0},\xi_{0})+O(1/R), &\quad &\eta\in B,\\
v(\eta)&=\gamma_{E}+h(z_{0},\xi_{0})+O(1/R), &\quad &\eta\in E,
\end{alignat*}
where the constant in $O(1/R)$ is independent of $z_0$.
To see this just write
$$
h(\eta,\xi)=(h(\eta,\xi)-h(\eta,\xi_{0}))+ (h(\xi_{0},\eta)-h(\xi_{0},z_{0}))+h(z_{0},\xi_{0}),
$$
use \eqref{eq10}, the symmetry of the Green's function and the fact that the equilibrium potential of a compact set is equal to the Robin constant on the set (except for an exceptional set of zero capacity).

\vspace*{7pt}

Now since $u= v=0$ on $\partial \Omega\backslash RQ$ one gets
\begin{align*}
u(z)&= \int_{\partial \tilde{\Omega}}u(\xi)\,d\omega (\tilde{\Omega},\xi,z)=\int_{\partial B}u(\xi)\,d\omega
(\tilde{\Omega},\xi,z),\\*[5pt]
v(z)&= \int_{\partial {\Omega}} v(\xi) \,d\omega ({\Omega},\xi,z)=\int_{\partial E}v(\xi)\,d\omega
({\Omega},\xi,z).
\end{align*}
Hence, for  $z\notin K\cup Q$,
\begin{align*}
u(z)&=(\gamma_{B}+h(z_{0},\xi_{0})+O (1/R))\, \omega (\tilde{\Omega},B,z),\\
v(z)&=(\gamma_{E}+h(z_{0},\xi_{0})+O (1/R))\, \omega (\Omega,E,z).
\end{align*}

Assume for the sake of simplicity that $\xi_{0}=0.$  Then by plugging the identity \eqref{eqgreen*} into
the above definitions of $u$ and $v$ we obtain
\begin{alignat*}{2}
u(z)&=\log \frac{1}{|z|} +\int_{B}h(z,\xi)\,d\mu_{B}(\xi),&\quad &z\notin B,\\*[5pt]
v(z)&=\int_{E} \log\frac{1}{|z-\xi|}\,d\mu_{E}(\xi) +\int_{E}h(z,\xi)\,d\mu_{E}(\xi),&\quad &z\notin E.
\end{alignat*}
Set
\begin{equation*}
\begin{split}
\varphi(z)&:=u(z)-v(z)=\int_{E}\left(\log \frac{1}{|z|}-\log\frac{1}{|z-\xi|}\right)\,d\mu_{E}(\xi)\\*[5pt]
&\quad+
\int_{B}h(z,\xi)\,d\mu_{B}(\xi)-\int_{E}h(z,\xi)\,d\mu_{E}(\xi)\\*[5pt]
&=\int_{E}\left(\log \frac{1}{|z|}-\log \frac{1}{|z-\xi|}\right)\,d\mu_{E}(\xi)+\int_{B}(h(z,\xi)-h(z,0))\,d\mu_{B}(\xi)\\*[5pt]
&\quad- \int_{E}(h(z,\xi)-h(z,0))\,d\mu_{E}(\xi).
\end{split}
\end{equation*}
Thus, for $z\in \Omega\backslash RQ$,
\begin{equation*}
\begin{split}
|\varphi(z)|&\le \left|\int_{E}\log \frac{|z-\xi|}{|z|}\,d\mu_{E}(\xi)\right|+\left|\int_{B}(h(z,\xi)-h(z,0))\,d\mu_{B}(\xi)\right| \\*[5pt]
&\quad+
\left|\int_{E}(h(z,\xi)-h(z,0))\right|\,d\mu_{E}(\xi)=O(1/R).
\end{split}
\end{equation*}
 We have used that for $\xi\in E$
$$
\log \frac{|z-\xi|}{|z|}\le \log \frac{|z|+|\xi|}{|z|}\le \log \left(1+\frac{2}{|z|}\right)=O\left(\frac{1}{|z|}\right)
$$
and
$$
\log \frac{|z|}{|z-\xi|}\le \log \left(1-\frac{|z-\xi|-|z|}{|z-\xi|}\right)\le \log \left(1+\frac{|\xi|}{|z-\xi|}\right)=O\left(\frac{1}{|z|}\right).
$$
Therefore
$$
u(z)=v(z)+O(1/|z|), \quad z\in \Omega\backslash RQ.
$$
Recalling that $\operatorname{Cap}(B)=\frac{1}{2}\operatorname{Cap}(E)^{1+\varepsilon}$ one gets
\begin{equation*}
\begin{split}
\omega(\tilde{\Omega},B,z)&=\frac{u(z)}{\gamma_{B}+h(z_{0},0)+O(1/R)}=\frac{v(z)+O(1/|z|)}{\gamma_{E}(1+\varepsilon)+\log 2+h(z_{0},0)+O(1/R)}\\*[7pt]
&=\frac{(\gamma_{E}+h(z_{0},0)+O(1/R))\, \omega (\Omega,E,z)+O(1/|z|)}
{\gamma_{E}(1+\varepsilon)+\log 2+h(z_{0},0)+O(1/R)}.
\end{split}
\end{equation*}
Clearly there exists $R_0(\varepsilon)$ such that for $R>R_0(\varepsilon)$ we have
\begin{equation*}
\begin{split}
\omega(\tilde{\Omega},B,z)&\ge \frac{1}{2}\frac{\gamma_{E}+h(z_{0},0)}{\gamma_{E}(1+\varepsilon)+\log 2+h(z_{0},0)}\, \omega (\Omega,E,z)
+ O\left(\frac{1}{|z|}\right),
\\*[5pt]
\end{split}
\end{equation*}
since the denominator $\gamma_{E}(1+\varepsilon)+\log 2+h(z_{0},0)$ is bounded below away from $0$  by~\eqref{hac}. Appealing again to \eqref{hac} we obtain that, for $R > R_0(\varepsilon)$,
$$
\frac{\gamma_{E}+h(z_{0},0)}{\gamma_{E}(1+\varepsilon)+\log 2+h(z_{0},0)} \geq \frac{1}{2}
$$
and so
$$
\omega(\tilde{\Omega},B,z) \ge \frac{1}{4} \omega (\Omega,E,z)
+ O\left(\frac{1}{|z|}\right).
$$
Letting $z \to\infty$ completes the proof of a) in the lemma.
 \end{proof}

\begin{proof}[Proof of b)]
Assume that $\xi_{0}=0$ and let $U=\left\{|z|< R \right\}$. The Green function $g=g_{U}$  of~$U$ is
$$
g(w,\xi)=\log \left|\frac{1- \frac{w}{R}  \, \frac{ \overline{\xi}}{R}}{\frac{w}{R}  - \frac{{\xi}}{R}}\right|.
$$
Let $g_{B}$ be the Green function of $U\backslash B$ and $g_{E}$ the Green function of  $U\backslash E$.  We claim that
\begin{equation}\label{gbe}
g_{B}(z,\xi)=g(z,\xi)-\int_{\partial B}g(w,\xi)\,d\omega (U\backslash B,w,z), \quad z, \xi \in U\backslash B.
\end{equation}
On one hand,  the right hand side $\phi(z,\xi)$ is a harmonic function of $z$ except for $z=\xi$ where it has a logarithmic pole.  On the other hand,
if $z$ tends to a point in $\partial (U\setminus B)$  then $\phi(z,\xi)$ tends to $0$, owing to the fact that $\int_{\partial B}g(w,\xi)\,d\omega (U\backslash B,w,z)$ is the solution of the Dirichlet problem in $U \setminus B$ with boundary values $g(z,\xi)$ with $\xi$ fixed.

Analogously one obtains
\begin{equation}\label{ge}
g_{E}(z,\xi)=g(z,\xi)-\int_{\partial E}g(w,\xi)\,d\omega (U\backslash E,w,z),\quad z, \xi \in U\backslash E.
\end{equation}
The goal is to prove the inequality
\begin{equation}\label{eq11}
\frac{\partial g_{B}}{\partial n}(z,\xi)\ge \frac{\partial g_{E}}{\partial n}(z,\xi),\quad |z|=\frac{R}{2},\quad \xi\in \partial U,
\end{equation}
which follows from
\begin{equation}\label{gmax}
g_{B}(z,\xi)\ge g_{E}(z,\xi),\quad |z|=\frac{R}{2},\quad   \frac{3}{4} R \le |\xi|  < R .
\end{equation}
Since $g_{B}(z,\xi) = g_{E}(z,\xi), \; |\xi|  = R, \;$ then, by the maximum principle, it is enough to show \eqref{gmax} for $|\xi|  = \frac{3}{4} R.$

We start by proving
\begin{equation}\label{glog}
\log \left(\frac{4}{3}\right) - \frac{C}{R} \le g(w, \xi) \le \log\left(\frac{4}{3}\right) + \frac{C}{R}, \quad |w|\le 1, \quad  |\xi|=\frac{3}{4} R,
\end{equation}
where $C$ is a positive constant and $R$ is sufficiently large.  We have
\begin{equation*}
\begin{split}
g(w, \xi) = &\log\left(\frac{4}{3}\right) + g(w, \xi)-g(0,\xi) = \log\left(\frac{4}{3}\right) + \log \left| 1-\frac{w \overline{\xi}}{R^2}\right| -\log\left|1-\frac{w}{\xi}\right|.
\end{split}
\end{equation*}
The absolute value of each of the last two terms is less than or equal to $C/R$ for some constant $C$ and \eqref{glog} follows.

Inserting \eqref{glog} into \eqref{gbe}  and \eqref{ge}  we get
\begin{alignat*}{3}
g_{B}(z,\xi)&\ge g(z,\xi)-\left(\log\left(\frac{4}{3}\right)+\frac{C}{R}\right)\, \omega (U\backslash B,B,z), &\quad & |z| = \frac{R}{2}, &\quad &|\xi|= \frac{3}{4}R, \\*[5pt]
g_{E}(z,\xi)&\le g(z,\xi)-\left(\log\left(\frac{4}{3}\right)-\frac{C}{R}\right)\, \omega (U\backslash E,E,z), &\quad & |z| = \frac{R}{2}, &\quad &|\xi|= \frac{3}{4}R.
\end{alignat*}

Clearly \eqref{gmax} is a consequence of the two preceding inequalities and the following claim.
\begin{claim}
For $R$ large enough one has 
$$
\left(\log\left(\frac{4}{3}\right)+\frac{C}{R}\right)\omega (U\backslash B,B,z)\le \left(\log\left(\frac{4}{3}\right)-\frac{C}{R}\right)\omega (U\backslash E,E,z),
\quad |z|=\frac{R}{2}.
$$
\end{claim}

We postpone the proof of the Claim and we proceed to complete the argument for Lemma \ref{lem1}.

Consider a subset $A$ of $\partial\Omega\backslash RQ.$ We want to prove
\begin{equation}\label{endb}
\omega(A,z)\le \tilde{\omega}(A,z), \quad |z|=\frac{R}{2},
\end{equation}
 where $\omega(A,z)$ stands for  $\omega(\Omega,A,z)$ and
$\tilde{\omega}(A,z)$ for $\omega(\tilde{\Omega},A,z)$. Take a point $z_0$ with $|z_{0}|=\frac{R}{2}$ such that
$$
\sup_{|z|=R/2}\frac{\omega(A,z)}{\tilde{\omega}(A,z)}=\frac{\omega(A,z_{0})}{\tilde{\omega}(A,z_0)}.
$$
Assume, to get a contradiction, that 
$\frac{\omega(A,z_{0})}{\tilde{\omega}(A,z_0)}=\lambda>1.$
Then
$$
\lambda \tilde{\omega}(A,z)-\omega(A,z)=\lambda-1>0, \quad  z\in A,
$$
and
$$
\lambda\tilde{\omega}(A,z)-\omega(A,z)\ge 0, \quad |z|=\frac{R}{2}.
$$
The maximum principle yields
$$
\lambda \tilde{\omega}(A,z)-\omega(A,z)>0, \quad z\in
\partial U.
$$
Since $\omega(A,\xi)$ is a harmonic function on $U \setminus E$
vanishing on $\partial E$ and, similarly,  $\tilde{\omega}(A,\xi)$
is a harmonic function on $U \setminus B$ vanishing on $\partial B$,
we get, by \eqref{eq11},
\begin{equation*}
\begin{split}
0&=\lambda \tilde{\omega}(A,z_0)-\omega(A,z_0)\\*[5pt]
&=\frac{1}{2\pi}\int_{\partial U}\frac{\partial g_{B}}{\partial
n}(z_{0},\xi)\lambda\,
\tilde{\omega}(A,\xi)\,ds(\xi)-\frac{1}{2\pi}\int_{\partial
U}\frac{\partial g_{E}}{\partial
n}(z_{0},\xi)\,\omega(A,\xi)\,ds(\xi) \\*[5pt]
& \ge \frac{1}{2\pi} \int_{\partial U}\frac{\partial  g_{B}}{\partial
n}(z_{0},\xi)\left(\lambda
\tilde{\omega}(A, \xi)-\omega(A, \xi)\right)\,ds(\xi)>0,
\end{split}
\end{equation*}
which is a contradiction. Then \eqref{endb} holds.
\vspace*{20pt}

\begin{center}
\includegraphics{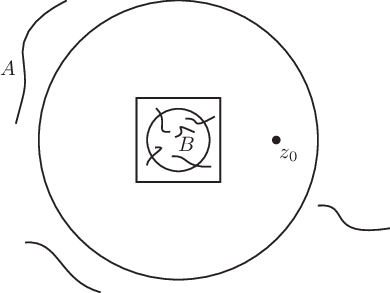}
\end{center}

\vspace*{17pt}

By \eqref{endb} and the maximum principle, 
$\omega(A,z)\le \tilde{\omega}(A,z)$ for $z \in \Omega$ and $|z|\ge
\frac{R}{2},$ and letting $|z|\to\infty$, item~b) of Lemma
\ref{lem1} follows.
\renewcommand{\qedsymbol}{}
\end{proof}

\begin{proof}[Proof of the Claim]
Recall that we are assuming $\ell(Q)=1$, so that for all compact
sets~$K$, $\operatorname{Cap}(E)=\operatorname{Cap}(K\cap Q)\le
1/\sqrt{2}$ and hence $\gamma_{E}\ge \log \sqrt{2} >0.$

Moreover
$$
\gamma_{B}=\gamma_{E}(1+\varepsilon)+\log 2>\gamma_{E}.
$$
Let $r=r(B)$ be the radius of~$B$. The function
$$
\log \left(\frac{R}{|z|}\right)\, \frac{1}{\log (R/r)}, \quad z \in U\setminus
B,
$$
is harmonic on $U\setminus B,$ vanishes on $|z|= R$ and is $1$
on $|z|= r$. Thus it is precisely $\omega(U \backslash B, B,z).$ Since $-\log
r(B)=\gamma_{B}$ we have
\begin{equation}\label{claimf}
 \omega (U\backslash B,B,z)=\log \left(\frac{R}{|z|}\right)\frac{1}{\log
R+\gamma_{B}}, \quad z \in U \setminus B.
\end{equation}

\vspace*{7pt}

We turn  now our attention on $\omega (U\backslash E,E,z)$. Consider the function
$$
f(z)=\int_{E}\log \frac{R}{|z-w|}\,d\mu_{E}(w)\, \frac{1}{\log
R+\gamma_{E}}\quad z \in U \setminus E.
$$
Since  $\int_{E}\log\frac{1}{|z-w|}\,d\mu_{E}(w)=\gamma_{E}$ for
$z\in E,$  except for a set of zero logarithmic capacity,  $f(z)=1$, $z\in E,$ except for a set of
zero logarithmic capacity.

If $w \in E$, $z\in \partial U$ one has $|z-w|=R+O(1)$ and so
$$
\log \frac{R}{|z-w|}=-\log \left(1-\frac{R-|z-w|}{R}\right)=-\log (1+O(1/R))=O(1/R)
$$
and 
$$\log \frac{|z-w|}{R}=-\log
\left(1-\frac{|z-w|-R}{|z-w|}\right)=-\log
\left(1+O(1/R)\right)=O(1/R).$$

Since $f(z)=1, \;\; z \in E$,  we conclude that
$$|f(z)|\le \frac{O(1/R)}{\log R+\gamma_{E}}, \quad z \in \partial U,$$
 so that the function
$$
\tilde{f}(z)=f(z)-\frac{C/R}{\log R+\gamma_{E}}
$$
satisfies $\tilde{f}(z)\le 1$, $z\in E,$ and $\tilde{f}(z)\le 0$,
$z\in \partial U,$ for an appropriate large constant~$C$. It follows
that
$$
\tilde{f}(z)\le\omega (U\backslash E,E,z),\quad z \in U \setminus E.
$$
To estimate this harmonic measure we write
\begin{equation*}
\begin{split}
\omega(U\backslash E,E,z)&\ge \frac{-C}{R(\log
R+\gamma_{E})}+\frac{1}{\log R+\gamma_{E}}\, \int_{E}\left(\log
\frac{R}{|z-w|}-\log \frac{R}{|z|}\right)\,d\mu_{E}(w)\\*[7pt]
&\quad +\frac{1}{\log R+\gamma_{E}}\, \log \frac{R}{|z|}=
T_1+T_2+T_3.
\end{split}
\end{equation*}
By \eqref{claimf}
$$
T_3=\frac{1}{\log R+\gamma_{B}}\, \log
\frac{R}{|z|}+\left(\frac{1}{\log R+\gamma_{E}}-\frac{1}{\log
R+\gamma_{B}}\right)\, \log \frac{R}{|z|}=\omega(U\backslash B,B,z)+
T_4.
$$
For the term $T_4$ we have
$$
T_4=\frac{\gamma_{B}-\gamma_{E}}{(\log R+\gamma_{E})(\log
R+\gamma_{B})}\, \log \frac{R}{|z|}\ge \frac{\varepsilon
\gamma_{E}+\log 2}{(\log R+2\gamma_{E}+\log 2)^{2}},
$$
provided $\varepsilon < 1,$ because $\gamma_{B}\le 2\gamma_{E}+\log
2$.

\noindent For the term $T_2$ we have
$$
|T_2|\le \frac{1}{\log R+\gamma_{E}}\int_{E}\left|\log
\frac{|z-w|}{|z|}\right|\,d\mu_{E}(\omega)
$$
with
$$
\log \frac{|z-w|}{|z|}=\log \left(1+\frac{|z-w|-|z|}{|z|}\right)=\log (1+O(1/R))=O(1/R)
$$
and the same estimate also holds for $\log \frac{|z|}{|z-w|}$.
Hence $$|T_2|\le \frac{C}{R(\log R+\gamma_{E})}.$$ 
Since $|T_1|$ obviously satisfies the same estimate, we conclude that
\begin{equation}\label{eq12}
\omega(U\backslash E,E,z)\ge \omega (U\backslash B,B,z)+\frac{\varepsilon
 \gamma_{E}+\log 2}{(\log R+2\gamma_{E}+\log 2)^{2}}-\frac{C}{R(\log
 R+\gamma_{E})},
\end{equation}
for some positive constant~$C$.

\vspace*{7pt}

 Recall that the claim is
$$
\left(\log\left(\frac{4}{3}\right)+\frac{C}{R}\right)\omega (U\backslash B,B,z)\le \left(\log\left(\frac{4}{3}\right)-\frac{C}{R}\right)\omega (U\backslash E,E,z),
\quad |z|=\frac{R}{2}.
$$
From now to the end of the proof of the claim $z$ denotes a point satisfying  $|z|=\frac{R}{2}.$   By~\eqref{eq12} we get, for $R \ge R_0(\varepsilon)$,
$$
\omega(U\setminus E,E,z)\ge \omega(U\setminus B,B,z)+C\frac{\varepsilon \gamma_{E}}{(\log R+\gamma_{E})^{2}}.
$$
It is sufficient to show
$$
\left(\log\left(\frac{4}{3}\right)\!+\frac{C}{R}\right)\omega(U\setminus B,B,z)\le \left(\log\left(\frac{4}{3}\right)-\frac{C}{R}\right)\!\left(\omega(U\setminus B,B,z)+C\frac{\varepsilon \gamma_{E}}{(\log R+\gamma_{E})^{2}}\right)\!
$$
or
$$
\frac{C \omega(U\setminus B,B,z)}{R}\le -\frac{C}{R}\omega(U\setminus B,B,z)+\left(\log\left(\frac{4}{3}\right)-\frac{C}{R}\right) C\frac{\varepsilon \gamma_{E}}{(\log R+\gamma_{E})^{2}},
$$
which amounts to, for $R \ge R_0(\varepsilon)$, 
$$
 \frac{\omega(U\setminus B,B,z)}{R} \le C\,\frac{\varepsilon \gamma_{E}}{(\log R+\gamma_{E})^{2}}.
$$
By \eqref{claimf}, for $|z|=R/2$,  we have
$$
\omega(U\setminus B,B,z)=\frac{2}{\log R+\gamma_{B}}=\frac{2}{\log R+(1+\varepsilon)\gamma_{E}+\log 2} \le \frac{2}{\log R + \gamma_E}.
$$
Then, for $R \ge R_0(\varepsilon)$, we get
$$
 \frac{\omega(U\setminus B,B,z)}{R} \le \frac{2}{R (\log R + \gamma_E)} \le C\, \frac{\varepsilon \gamma_E }{(\log R + \gamma_E)^2},
$$
where the last inequality is equivalent to
$$
2 (\log R + \gamma_E ) \le C\, R \varepsilon \gamma_E,
$$
which is clearly true for $R$ large enough, because $\gamma_E \ge \log \sqrt{2}.$
\end{proof}

\subsection{Proof of Lemma~\ref{lem3}}

We note that in the statement of Lemma~\ref{lem3} one has to understand that no curve~$\Gamma_{j}$ lies inside another curve~$\Gamma_{k}$; in other words, the bounded connected components of $\mathbb{C} \setminus \Gamma_j,$ \; $1 \le j \le N,$ are disjoint.  Also, replacing $K$ by $\{g\le \varepsilon\}$ for small $\varepsilon>0$, we can assume $\Omega$ is a finitely connected domain with smooth boundary.

\vspace*{7pt}

Recall that we can write the Green function~$g$ as
\begin{equation}\label{greeninfinity}
g(z)=\log |z|+\gamma_{K}+h_{0}(z),
\end{equation}
where 
$$h_{0}(z)=\frac{a_{1}}{z}+\frac{\bar a_{1}}{\bar z}+\frac{a_{2}}{z^{2}}+\frac{\bar a_{2}}{\bar z^{2}}+\dotsb$$
is harmonic and satisfies $h_{0}(\infty)=0$.

\vspace*{7pt}

Let $\{\xi_{k}\}$ be the set of critical points of~$g$ that lie outside~$\Gamma$. First of all we note that there is only a finite number of these critical points. Indeed, the $\xi_{k}$'s are the zeros of~$\partial g$, which is a holomorphic function on $\Omega=\mathbb{C}^* \setminus K$ vanishing at infinity.  Hence the critical points can accumulate only on~$K$ and so outside~$\Gamma$ there are only finitely many, say $\xi_1, \dots, \xi_L$.

\begin{center}
\includegraphics{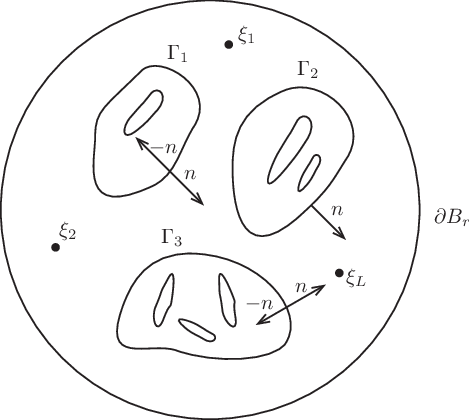}
\end{center}

We want to show the equality
$$
\frac{1}{2\pi}\int_{\Gamma}\frac{\partial g}{\partial
n}\, \log |\nabla g|\,ds=\sum^{L}_{k=1}g
(\xi_{k})+\sum^{N}_{j=1}c_{j}\frac{1}{2\pi}\int_{\Gamma_{j}}\frac{\partial}{\partial
n}\, \log |\nabla g|\,ds+\gamma_{K}.
$$
Let $B_{r}$ be the disc centered at the origin of radius~$r$ big enough to contain the unit disc and all the critical points of~$g$. Green's formula gives
\begin{multline*}
-\frac{1}{2\pi}\int_{\Gamma}\frac{\partial g}{\partial n}\, \log |\nabla g|\,ds +\frac{1}{2\pi}\int_{\partial B_{r}}\frac{\partial g}{\partial n}\, \log |\nabla g|\,ds \\
=-\frac{1}{2\pi}\int_{\Gamma}g\frac{\partial}{\partial n}\, \log
|\nabla g|\,ds +\frac{1}{2\pi}\int_{\partial B_{r}}g\,
\frac{\partial}{\partial n}\, \log |\nabla g|\,ds
-\sum^{L}_{k=1}g(\xi_{k}),
\end{multline*}
where we used that $\Delta \log  |\partial g  | = 2 \pi \sum_{k=1}^L  \delta_{\xi_k}.$
Equivalently
\begin{equation*}
\begin{split}
\frac{1}{2\pi}\int_{\Gamma}\frac{\partial g}{\partial
n}\, \log |\nabla g|\,ds&=\sum^{L}_{k=1}g(\xi_{k})+\sum^{N}_{j=1}c_{j}\, \frac{1}{2\pi}\int_{\Gamma_{j}}\frac{\partial}{\partial n}\, \log |\nabla g|\,ds\\*[5pt]
&\quad+\frac{1}{2\pi}\, \int_{\partial B_{r}}\left( \frac{\partial g}{\partial n}\, \log |\nabla g|-g \frac{\partial}{\partial n}\, \log |\nabla g|\right)\,ds
\end{split}
\end{equation*}
and we need to prove
\begin{equation}\label{eq13}
\lim_{r\to\infty}\frac{1}{2\pi} \int_{\partial B_{r}}\left(\frac{\partial g}{\partial n}\log |\nabla g|-g\frac{\partial}{\partial n}\log |\nabla g|\right)\,ds=\gamma_{K}.
\end{equation}
On $\partial B_{r}$ the normal derivative $\frac{\partial}{\partial n}$ is the partial derivative~$\frac{\partial}{\partial r}.$  By \eqref{greeninfinity}  
$$\frac{\partial g}{\partial r} (z)= \frac{1}{r} + \frac{\partial}{\partial r}h_{0}(z) = \frac{1}{r} + O\left(\frac{1}{r^2}\right), \quad |z|=r,$$
 and similarly
 $$\nabla g(z)= \nabla \log |z| + \nabla h_{0}(z) = \frac{1}{\bar z} + O\left(\frac{1}{|z|^2}\right), \quad |z|=r.$$
Thus
 $$ \log |\nabla g(z)| = \log \frac{1}{r}+O\left(\frac{1}{r}\right)$$ 
 and
 $$\frac{\partial}{\partial r}\log |\nabla g(z)| = -\frac{1}{r}+ O\left(\frac{1}{r^2}\right).$$
 
The integral in \eqref{eq13} becomes comparable to
\begin{equation*}
\begin{split}
\frac{1}{2\pi}\int_{\partial B_{r}}\left(\frac{1}{r}\log \frac{1}{r}+(\log r+h_{0}+\gamma_{K})\, \frac{1}{r}\right)\,ds + O\left(\frac{1}{r}\right)\!=\!
(h_{0}+\gamma_{K}) + O\left(\frac{1}{r}\right),
\end{split}
\end{equation*}
which tends to $\gamma_{K}$ as $r \rightarrow \infty,$
because $h_{0}(r)\to 0$.

\vspace*{7pt}

The next step is to prove the identities
$$
\frac{1}{2\pi} \int_{\Gamma_{j}}\frac{\partial}{\partial n}\log |\nabla g|\,ds=-1,\quad j=1,2,\dotsc,N.
$$
Since $\nabla g = 2 \bar\partial,$
\begin{equation*}
\begin{split}
\frac{1}{2\pi}\int_{\Gamma_{j}}\frac{\partial}{\partial n}\log |\nabla g|\,ds &= \frac{1}{2\pi}\int_{\Gamma_{j}}\frac{\partial}{\partial n}\, \log |\bar\partial g|\,ds\\*[7pt]
&=\frac{1}{2\pi}\int_{\Gamma_{j}}\langle 2\bar\partial \log |\bar\partial g|,n\rangle\,ds=\frac{1}{2\pi}\int_{\Gamma_{j}}\left\langle \frac{\bar\partial^{2}g}{\bar\partial g},n\right\rangle\,ds\\*[7pt]
&=\operatorname{Re}\left(\frac{1}{2\pi}\int_{\Gamma_{j}}\frac{\partial^{2}g}{\partial g}\, n\,ds\right)
=\operatorname{Re}\left(\frac{1}{2\pi i}\int_{\Gamma_{j}}\frac{\partial^{2}g}{\partial g}\,dz\right)\\*[7pt]
&=\frac{1}{2\pi}\operatorname{Var}\operatorname{arg}_{\Gamma_{j}}(\partial g)=
\frac{1}{2\pi}\operatorname{Var}\operatorname{arg}_{\Gamma_{j}}(\overline{\nabla g})=-1.
\end{split}
\end{equation*}
Therefore
$$
\frac{1}{2\pi}\int_{\Gamma}\frac{\partial g}{\partial
n}\, \log |\nabla g|\,ds=\sum^{L}_{k=1}g(\xi_{k})-\sum^{N}_{j=1}c_{j}+\gamma_{K}
$$
and the proof of the lemma is reduced to 
\begin{equation}\label{eq14}
\sum^{N}_{j=1}c_{j}\le \sum^{L}_{k=1}g(\xi_{k})+\gamma_{K}+\log 2.
\end{equation}
Let $\mu_{K}$  be  equilibrium measure of~$K.$  Then 
$$
g(z)=\gamma_{K}+\int_{K}\log |z-w|\,d\mu_{K}(w),
$$
and so, recalling that $\Gamma\subset \{|z|<1\}$, 
\begin{equation}\label{eq15}
g(z)\le \gamma_{K}+\log 2,  \quad  |z|<1.
\end{equation}

Now we make a remark.
Let $\gamma$ be a Jordan curve which is contained in a level set of~$g$ and that surrounds a number~$\beta$ of connected components of~$K.$  Then the number of critical points of~$g$ inside~$\gamma$ is $\beta-1.$ 

\begin{center}
\includegraphics{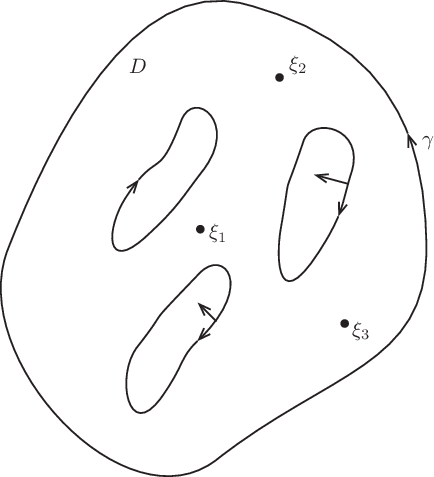}
\end{center}
\vspace*{17pt}

To see this,  let~$D$ stand for the domain bounded by $\gamma$ and~$K$; then $\nabla g$ is orthogonal to the boundary of~$D$ and when we travel along $\partial D$ the argument of~$\nabla g$ increases by~$2\pi$ over~$\gamma$ and decreases by~$2\pi$ over the boundary of each component of~$K$. So the total variation of $\operatorname{arg}(\partial g)$ on~$\partial D$ is $(\beta-1)\, 2\pi$ and, by the argument principle, $\partial g$ has $\beta-1$~zeros in~$D$.

\vspace*{7pt}

Take now $\gamma$ containing all critical points of $g$ and $K.$  Then the total number of critical points of $g$ is the number of components of $K$ minus $1$. Assuming that $\Gamma_{j}$ contains $\beta_{j}$ components of~$K$, $j=1,\dotsc,N$, we know that the number of critical points inside $\gamma_j$ is $\beta_j-1$ and so the number of critical points outside $\Gamma$ is $N-1.$  Replacing in~\eqref{eq14} the number $L$ of critical points outside $\Gamma$ by $N-1$, the inequality to be proven is
\begin{equation}\label{eq16}
\sum^{N}_{j=1}c_{j}\le \sum^{N-1}_{k=1}g(\xi_{k})+\gamma_{K}+\log 2.
\end{equation}

To show \eqref{eq16} let us assume that the constants~$c_{j}$ are different and ordered so that $c_{1}<c_{2}<\dotsb <c_{N}.$  We would like to understand how the $N-1$ critical points outside~$\Gamma$ appear.

\begin{center}
\includegraphics{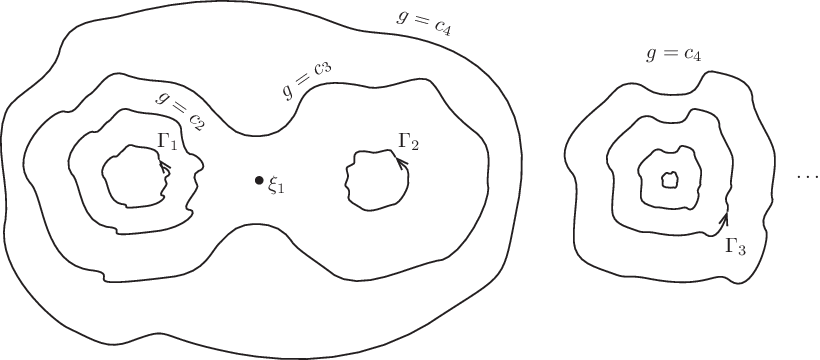}
\end{center}

\noindent

\begin{center}
\includegraphics{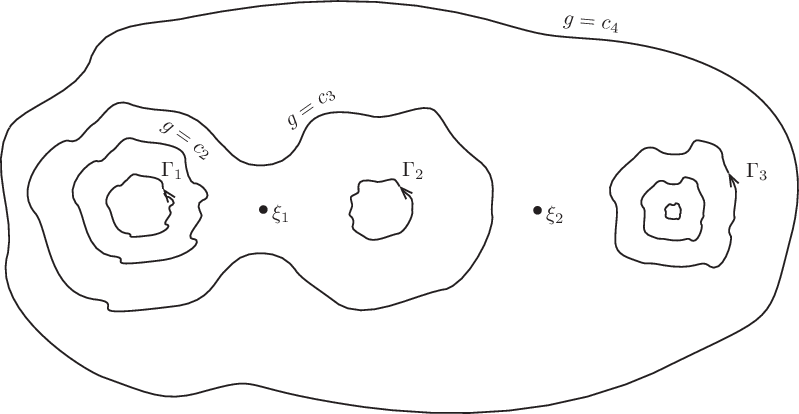}
\end{center}

The critical points of~$g$ appear when two components of a level set of~$g$ touch.  The critical point may have a multiplicity if more than two components coincide at a point; in this case, the multiplicity is, by the argument principle, the number of components that are joining minus one.  Assume, for instance, that two components of $\{z: g(z)=c \}$ intersect at $\xi_1$ and $c$ is the least number with this property. On one hand, $\nabla g(\xi_1) =0$, since otherwise $\{z: g(z)=c \}$ would be a smooth curve around~$\xi_1,$ which is not the case. On the other hand, the domain bounded by $\{z: g(z)=c \}$ contains two $\Gamma_j$,  which must be $\Gamma_1$ and $\Gamma_2$. Thus $g(\xi_1) \ge c_2.$ 
If there were three components of $\{z: g(z)=c \}$  which join at $\xi_1,$ then  $\Gamma_1$, $\Gamma_2$ and $\Gamma_3$ would be inside the domain bounded by $\{z: g(z)=c \}$. Hence $g(\xi_1) \ge c_3.$ 
Arguing inductively in this way we finally obtain that the $N-1$ critical points of $g$ outside $\Gamma$ satisfy
$$
\sum^{N-1}_{k=1}g(\xi_{k})\ge \sum^{N}_{j=2}c_{j}.
$$
Since $c_{1}=g(\tau)$ for some $\tau$,  \eqref{eq15}  gives $c_{1}\le \gamma_{K}+\log 2$  and \eqref{eq16} follows.\qed

\subsection{Proof of Lemma~\ref{lem4}}

Let $\mu$ be a Frostman measure for~$E$ with respect to the measure function~$h(r)=r^{s}$; that is, $\mu$ satisfies
$$
\mu (E)\approx M_{h}(E)\quad \text{and}\quad \mu (B(z,r))\le r^{s},\quad z \in \mathbb{C}, \quad r>0.
$$
The logarithmic potential of $\mu$ at a point~$z\in E$ is
\begin{equation*}
\begin{split}
U_{\mu}(z) & =\int_{E}\log \frac{1}{|z-w|}\,d\mu (w)=\int_{0}^{\infty}\mu \left\{w: \log \frac{1}{|z-w|}>t\right\}\,dt \\*[5pt]
& \le  \int_{0}^\tau \mu(E) \,dt + \int_{\tau}^\infty  e^{-st} \,dt = \tau \mu(E) + \frac{e^{-\tau s}}s.
\end{split}
\end{equation*}
The value of  $\tau$ which makes minimal the above expression is $\tau s = -\log \mu(E)$ and thus
\begin{equation*}
\begin{split}
U_{\mu}(z)& \le  \frac{\mu (E)}s \left(1+  \log \frac{1}{\mu(E)}\right), \quad z \in \mathbb{C}.
\end{split}
\end{equation*}
Recall now that the Robin constant of $E$ equals
$$
\gamma_{E} = \inf_\nu \int U_\nu(z)\,d\nu(z),$$
where the infimum is taken over all Borel probability measures supported on $E$.
So from the above calculation and the choice $\nu= \mu/\mu(E)$ we get the estimate
$$\gamma_E 
\le \frac{1}{s} \left(1+  \log \frac{1}{\mu(E)}\right).$$
Therefore
$$
\operatorname{Cap}(E)=e^{-\gamma_{E}}\ge C\,\mu (E)^{1/s}.\qedif
$$

\vspace*{7pt}

\begin{tabular}{@{}l}
Departament de Matem\`{a}tiques\\
Universitat Aut\`{o}noma de Barcelona\\
08193 Bellaterra, Barcelona, Catalonia\\
{\it E-mail:} {\tt jcufi@mat.uab.cat}\\
{\it E-mail:} {\tt xtolsa@mat.uab.cat}\\
{\it E-mail:} {\tt jvm@mat.uab.cat}
\end{tabular}

\end{document}